\documentclass[12pt]{amsart}
\usepackage{amssymb, fullpage}
\usepackage{enumerate}
\usepackage[all]{xy}
\makeatletter
\@namedef{subjclassname@2010}{%
  \textup{2010} Mathematics Subject Classification}
\makeatother
\newtheorem{thm}{Theorem}[section]
\newtheorem{cor}[thm]{Corollary}
\newtheorem{lemma}[thm]{Lemma}
\newtheorem{prop}[thm]{Proposition}

\newtheorem{mainthm}[thm]{Theorem}
\theoremstyle{definition}
\newtheorem{defn}[thm]{Definition}
\newtheorem{rem}[thm]{Remark}
\newtheorem{ex}[thm]{Example}
\numberwithin{equation}{section}

\def\R{{\mathbb{R}}}
\def\N{{\mathbb{N}}}
\def\C{{\mathbb{C}}}

\def\Z{{\mathbb{Z}}}
\def\I{{\mathbb{I}}}
\def\D{{\mathbb{D}}}
\def\S{{\mathbb{S}}}
\newcommand{\T}{\mathbb{T}}
\newcommand{\surj}{\text{surj}}
\newcommand{\inj}{\text{inj}}

\begin{document}
\title[Homotopical Stable Ranks]{Homotopical Stable Ranks for Certain C*-Algebras}
\author[P. Vaidyanathan]{Prahlad Vaidyanathan}
\address{Department of Mathematics\\ Indian Institute of Science Education and Research Bhopal\\ Bhopal ByPass Road, Bhauri, Bhopal 462066\\ Madhya Pradesh. India.}
\email{prahlad@iiserb.ac.in}
\date{}

\maketitle


\begin{abstract}
We study the general and connected stable ranks for C*-algebras. We estimate these ranks for pullbacks of C*-algebras, and for tensor products by commutative C*-algebras. Finally, we apply these results to determine these ranks for certain commutative C*-algebras and non-commutative CW-complexes.
\end{abstract}

\maketitle


Stable ranks for C*-algebras were first introduced by Rieffel \cite{rieffel} in his study of the nonstable K-theory of noncommutative tori. A stable rank of a C*-algebra $A$ is a number associated to the C*-algebra, and is meant to generalize the notion of covering dimension for topological spaces. The first such notion introduced by Rieffel, called \emph{topological stable rank}, has played an important role ever since. In particular, the structure of C*-algebras having topological stable rank one is particularly well understood. \\

Since the foundational work of Rieffel, many other ranks have been introduced for C*-algebras, including real rank, decomposition rank, nuclear dimension, etc. In this paper, we return to the original work of Rieffel, and consider two other stable ranks introduced by him: the \emph{connected stable rank} and \emph{general stable rank}. The general stable rank determines the stage at which stably free projective modules are forced to be free. The connected stable rank is a related notion, but its definition is less transparent. What links these two ranks, and differentiates them from the topological stable rank, is that they are homotopy invariant. \\

This was highlighted in a paper by Nica \cite{nica2}, where he emphasized the relationship between these two ranks, and how they differ from topological stable rank. Furthermore, in order to compute these ranks for various examples, he showed how they behave with respect to some basic constructions (matrix algebras, quotients, inductive limits, and extensions). \\

The goal of this paper is to extend these results by examining how the connected and general stable rank (together referred to as \emph{homotopical} stable ranks) behave with respect to iterated pullbacks and tensor products with commutative C*-algebras, and thereby calculate these ranks for some familiar C*-algebras. We mention here that the homotopical stable ranks play an important role in K-theory (\cite{rieffel2}, \cite{sheu}, \cite{xue}). Although we do not dwell on that much, we believe that further investigations into these ranks will yield a much better understanding of nonstable phenomena in K-theory.\\

We now describe our results. Henceforth, we write $tsr, gsr$, and $csr$ to denote the topological, general and connected stable ranks respectively. To begin with, consider a pullback diagram of unital C*-algebras
\[
\xymatrix{
A\ar[r]\ar[d] & B\ar[d]^{\delta} \\
C\ar[r]^{\gamma} & D
}
\]
where either $\gamma$ or $\delta$ is surjective. Note that $tsr(A)$ can be estimated by a theorem of Brown and Pedersen \cite[Theorem 4.1]{brown}
\[
tsr(A)\leq \max\{tsr(B),tsr(C)\}
\]
However, simple examples (See Example \ref{ex:sr_pullback}) show that the analogous estimate for homotopical stable ranks cannot hold. Instead, we show that
\begin{mainthm}\label{thm:main_pullback}
Given a pullback diagram as above,
\begin{equation*}
\begin{split}
gsr(A) &\leq \max\{csr(B),csr(C),gsr(C(\T)\otimes D))\}, \text{ and } \\
csr(A) &\leq \max\{csr(B),csr(C), csr(C(\T)\otimes D))\}
\end{split}
\end{equation*}
Furthermore, if $K_1(D) = 0$, the first estimate may be improved to
\[
gsr(A)\leq \max\{gsr(B),gsr(C),gsr(C(\T)\otimes D))\}
\]
\end{mainthm}
We should mention here that the precise estimates are slightly finer than those mentioned above (See Proposition \ref{prop:inj_surj_bounds}), and they depend on specific information about the homotopy groups of the groups $GL_n(D)$ of invertible matrices over $D$. \\

We turn to tensor products by commutative C*-algebras. If $Y$ is a compact Hausdorff space, a projective module over $C(Y)$ corresponds to a vector bundle over $Y$. Furthermore, if $Y = \Sigma X$, the reduced suspension of $X$, then any vector bundle over $Y$ of rank $n$ corresponds to the homotopy class of a map from $X$ to $GL_n(\C)$. Building on work of Rieffel \cite{rieffel3}, we describe all projective modules over C*-algebras of the form $C(\Sigma X)\otimes A$ in an analogous fashion. We conclude that
\begin{mainthm}\label{thm:main_gsr_csxa}
For a compact Hausdorff space $X$ and a unital C*-algebra $A$,
\[
gsr(C(\Sigma X)\otimes A) = \max\{gsr(A), \inj_X(A)\}
\]
where $\inj_X(A)$ denotes the least $n\geq 1$ such that the map $[X,GL_{m-1}(A)]_{\ast}\to [X,GL_m(A)]_{\ast}$ is injective for all $m\geq n$.
\end{mainthm}

We also obtain various estimates for $csr(C(X)\otimes A)$ as well which, once again, depend on the homotopy groups of $GL_n(A)$. These estimates are particularly sharp in the case where the natural map $GL_{n-1}(A)\to GL_n(A)$ is a weak homotopy equivalence. In that case, we explicitly determine the homotopical stable ranks of $C(X)\otimes A$ in terms of those of $A$ (Theorem \ref{thm:gsr_csr_cxa}). \\

Finally, we apply these results to a variety of examples. In particular, using Theorem \ref{thm:main_gsr_csxa}, we determine $gsr(C(\T^d))$ (Example \ref{ex:gsr_td}), thus answering a question of Nica \cite[Problem 5.8]{nica2}. We also estimate the homotopical stable ranks for noncommutative CW-complexes (Theorem \ref{thm:nccw_complex}), which naturally fall into the ambit of this paper.


\section{Preliminaries}
\subsection{Stable Ranks}\label{subsec:stable_ranks}
Let $A$ be a unital C*-algebra, and $n\in \N$. A vector $\underline{a} = (a_1,a_2,\ldots, a_n) \in A^n$ is said to be \emph{left unimodular} if $Aa_1 + Aa_2 + \ldots + Aa_n = A$. Equivalently, $\underline{a}$ is left unimodular if $\exists \underline{a'} = (a_1',a_2',\ldots, a_n') \in A^n$ such that $\sum_{i=1}^n a_i'a_i = 1_A$. We write $Lg_n(A)$ for the set of all left unimodular vectors in $A^n$. There is an analogous notion of a right unimodular vector, but the continuous involution on $A$ induces a homeomorphism between the two sets. For this reason, we only focus on the left unimodular vectors. \\

Unimodular vectors are related to projective modules by the following observation (See \cite[Lemma 10.4]{rieffel}): If $\underline{a} \in A^n$, then $\underline{a} \in Lg_n(A)$ if and only if $\underline{a}A$ is a direct summand of $A^n$.
\begin{equation}\label{eqn:stably_free}
A^n \cong P\oplus \underline{a}A
\end{equation}
Conversely, if $P$ is a projective right $A$-module such that $P\oplus A\cong A^n$, then there is a left unimodular vector $\underline{a}\in Lg_n(A)$ such that Equation \ref{eqn:stably_free} holds. The most interesting fact in all this is that this vector $\underline{a}$ also determines when $P$ is itself a free module.\\

Let $GL_n(A)$ denote the set of invertible elements in $M_n(A)$. Note that $GL_n(A)$ acts on $Lg_n(A)$ by left multiplication: $(T,\underline{a}) \mapsto T(\underline{a})$. Let $e_n \in Lg_n(A)$ denote the vector $(0,0,\ldots, 0,1)$. Now, if $P$ is a projective right $A$-module such that Equation \ref{eqn:stably_free} holds, then $P\cong A^{n-1}$ if and only if $\exists T\in GL_n(A)$ such that $T(\underline{a}) = e_n$ (See, for instance, \cite[Proposition 4.14]{magurn}). This leads to the following definition

\begin{defn}\label{defn:gsr}
Let $A$ be a unital C*-algebra, then the general stable rank (gsr) of $A$ is the least $n\geq 1$ such that either (and hence both) of the following hold:
\begin{itemize}
\item $GL_m(A)$ acts transitively on $Lg_m(A)$ for all $m\geq n$
\item For all $m\geq n$, if $P$ is a projective module over $A$ such that $P\oplus A\cong A^m$, then $P\cong A^{m-1}$
\end{itemize}
\end{defn}
If no such $n$ exists, we simply write $gsr(A) = +\infty$. To avoid repetition, we will adopt the same convention in the definitions of connected and topological stable rank below. \\

Recall that $A$ has the \emph{invariant basis number} (IBN) property if $A^m \cong A^n$ implies that $m=n$. This is equivalent to requiring that $[A]$ has infinite order in $K_0(A)$. Now a swindle argument (See \cite[Corollary 4.22]{magurn}) shows that if $A$ does not have the IBN property, then $gsr(A) = +\infty$. Thus, in this paper, we will be concerned only with C*-algebras having this property. \\

Given a C*-algebra $A$ that has the IBN property, any projective module $P$ satisfying the condition $P\oplus A^m \cong A^n$ (ie. \emph{stably free} projective modules) may be assigned a rank $(n-m)$, which is independent of the isomorphism. Hence, the general stable rank of $A$ simply determines the least rank at which stably free projective modules are forced to be free. \\

Now, the first condition of Definition \ref{defn:gsr} leads to another observation: Any subgroup of $GL_n(A)$ also acts on $Lg_n(A)$. Let $GL_n^0(A)$ denote the connected component of the identity in $GL_n(A)$, and let $El_n(A)$ denote the subgroup of $GL_n(A)$ generated by elementary matrices, ie. matrices which differ from the identity matrix by atmost one off-diagonal entry. Note that $El_n(A)$ is a subgroup of $GL_n^0(A)$. It was proved by Rieffel \cite[Corollary 8.10]{rieffel} that, for $n\geq 2$, $GL_n^0(A)$ acts transitively on $Lg_n(A)$ if and only if $El_n(A)$ acts transitively on $Lg_n(A)$. Furthermore, he observed that the least $n$ at which this occurs also has a topological interpretation \cite[Corollary 8.5]{rieffel}, given as the second condition below
\begin{defn}\label{defn:csr}
Let $A$ be a unital C*-algebra, then the connected stable rank (csr) of $A$ is the least $n\geq 1$ such that either (and hence both) of the following hold:
\begin{itemize}
\item $GL_m^0(A)$ acts transitively on $Lg_m(A)$ for all $m\geq n$

\item $Lg_m(A)$ is connected for all $m\geq n$.
\end{itemize}
For $n\geq 2$ this is equivalent to the condition
\begin{itemize}
\item $El_m(A)$ acts transitively on $Lg_m(A)$ for all $m\geq n$
\end{itemize}
\end{defn}
We now turn to the notion of stable rank that has proved to be most useful in applications.

\begin{defn}
Let $A$ be a unital C*-algebra. Then the topological stable rank (tsr) of $A$ is the least $n\geq 1$ such that $Lg_n(A)$ is dense in $A^n$. 
\end{defn}

We mention here that if $Lg_n(A)$ is dense in $A^n$, then $Lg_m(A)$ is dense in $A^m$ for all $m\geq n$. However, the analogous statements are not true with respect to Definition \ref{defn:gsr} and \ref{defn:csr}. Indeed, it is possible that $GL_n(A)$ acts transitively on $Lg_n(A)$, but $GL_{n+1}(A)$ does not act transitively on $Lg_{n+1}(A)$. For instance, if $A$ is a finite C*-algebra, then $GL_1(A) = Lg_1(A)$, so $GL_1(A)$ clearly acts transitively on $Lg_1(A)$, but it is not true that $gsr(A) = 1$ when $A$ is finite. 

\begin{rem}\label{rem:properties}
If $A$ is a non-unital C*-algebra, then the stable rank of $A$ is simply defined as the stable rank of $A^+$, the C*-algebra obtained by adjoining a unit to $A$.
\end{rem}

We now enumerate some basic properties of these ranks that are known or are easily observed from the definitions. While the original proofs are scattered through the literature, \cite{nica2} is an immediate reference for all these facts.

\begin{enumerate}
\item $gsr(A\oplus B) = \max\{gsr(A),gsr(B)\}$. Analogous statements hold for $csr$ and $tsr$.
\item $gsr(A) \leq csr(A) \leq tsr(A) + 1$.
\item[] Strict inequalities are possible in both cases. In fact, it is possible that $tsr(A) = +\infty$, while $csr(A) < \infty$.
\item For any $n\in \N$,
\begin{gather*}
csr(M_n(A))\leq \left\lceil \frac{csr(A)-1}{n}\right\rceil + 1, \text{ and } gsr(M_n(A))\leq \left\lceil \frac{gsr(A)-1}{n}\right\rceil + 1
\end{gather*}
Here, $\lceil x\rceil$ refers to the least integer greater than or equal to $x$.
\item If $\pi : A\to B$ is surjective, then
\begin{gather*}
csr(B)\leq \max\{csr(A),tsr(A)\}, \text{ and } gsr(B)\leq \max\{gsr(A),tsr(A)\}
\end{gather*}
\item Furthermore, if $\pi :A\to B$ is a split epimorphism (ie. there is a morphism $s:B\to A$ such that $\pi\circ s = \text{id}_B$), then
\begin{gather*}
csr(B)\leq csr(A), \text{ and } gsr(B)\leq gsr(A)
\end{gather*}
\item If $0\to J\to A\to B$ is an exact sequence of C*-algebras, then
\begin{gather*}
csr(A) \leq \max\{csr(J),csr(B)\}, \text{ and } gsr(A) \leq \max\{gsr(J),csr(B)\}
\end{gather*}
It is worth mentioning here that when $J$ is an ideal of $A$, then there is, a priori, no relation between the homotopical stable ranks of $A$ and those of $J$.
\item If $\{A_i : i\in J\}$ be an inductive system of C*-algebras with $A := \lim A_i$, then
\begin{gather*}
csr(A) \leq \liminf_i csr(A_i), \text{ and } gsr(A) \leq \liminf_i gsr(A_i)
\end{gather*}
\item If $gsr(A) = 1$ (and hence if $csr(A) = 1$), then $A$ is stably finite. Conversely, if $gsr(A)\leq 2$ and $A$ is finite, then $gsr(A) = 1$.
\item If $csr(A) = 1$, then $K_1(A) = 0$. The converse is true if $tsr(A) = 1$.
\item If $tsr(A) = 1$, then $A$ has cancellation of projections, so $gsr(A) = 1$.
\end{enumerate}
Finally, we turn to the most interesting property shared by gsr and csr, viz. homotopy invariance. Two morphisms $\phi_0, \phi_1 : A\to B$ are said to be \emph{homotopic} if there is a $\ast$-homomorphism $h:A\to C([0,1],B)$ such that $\varphi_0 = p_0\circ h$ and $\varphi_1 = p_1\circ h$, where $p_t(\zeta) := \zeta(t)$. In symbols, we denote this by $\varphi_0\simeq \varphi_1$. We say that $A$ \emph{homotopically dominates} $B$ if there are morphisms $\varphi : A\to B$ and $\psi : B\to A$ such that $\varphi\circ\psi\simeq \text{id}_B$. If, in addition, $\psi\circ \varphi\simeq \text{id}_A$, then we say that $A$ and $B$ are homotopy equivalent (in symbols, $A\simeq B$). In the commutative case, $C(X)\simeq C(Y)$ if and only if $X$ and $Y$ are homotopy equivalent as topological spaces (once again, we write $X\simeq Y$ if this happens). The following result is due to Nistor \cite[Lemma 2.8]{nistor} for the connected stable rank and Nica \cite[Theorem 4.1]{nica2} for the general stable rank.
\begin{thm}\label{thm:homotopy_dom}
If $A$ homotopically dominates $B$, then
\begin{gather*}
csr(A)\geq csr(B), \text{ and }  gsr(A) \geq gsr(B)
\end{gather*}
In particular, if $A \simeq B$, then $csr(A) = csr(B)$ and $gsr(A) = gsr(B)$.
\end{thm}

We now turn to the problem of computing these ranks. An important tool in such an investigation is the following (See \cite[Section 1]{rieffel2}): For $m\geq 2$, the orbit of $e_m\in Lg_m(A)$ under the action of $GL_m(A)$ is called the space of \emph{last columns} of $A$, and is denoted by $Lc_m(A)$. It was first proved by Corach and Larotonda \cite{corach} that the natural map $t: GL_m(A) \to Lc_m(A)$ defines a principal, locally trivial fiber bundle on $Lc_m(A)$, with structural group $TL_m(A)$, the set of matrices of the form
\[
\begin{pmatrix}
x & 0 \\
c & 1
\end{pmatrix}
\]
where $x\in GL_{m-1}(A)$ and $c\in A^{m-1}$. Now, $TL_m(A)$ is homotopy equivalent to $GL_{m-1}(A)$, so the long exact sequence of homotopy groups arising from the fibration $TL_m(A)\to GL_m(A)\xrightarrow{t} Lc_m(A)$ takes the form
\begin{equation}\label{eqn:les_fibration}
\ldots \to \pi_{n+1}(Lc_m(A))\to \pi_n(GL_{m-1}(A))\to \pi_n(GL_m(A))\to \pi_n(Lc_m(A))\to \ldots
\end{equation}
which ends in a sequence of pointed sets $\pi_0(GL_{m-1}(A))\to \pi_0(GL_m(A))\to \pi_0(Lc_m(A))$. This will be of fundamental importance to us.

\subsection{Notational Conventions}

We fix some notation we will use repeatedly throughout the paper: We write $\S^n$ for the $n$-dimensional sphere, $\D^n$ for the $n$-dimensional disk, $\I^k$ for the $k$-fold product of the unit interval $\I=[0,1]$, and $\T^k$ for the $k$-fold product of the circle $\T$. Given a C*-algebra $A$ and a compact Hausdorff space $X$, we identify $C(X)\otimes A$ with $C(X,A)$, the space of continuous functions taking values in $A$. If $X = \T^k$, we simply write $\T^kA$ for $C(\T^k,A)$. We write $\theta_A^n$ for the map $GL_{n-1}(A) \to GL_n(A)$ given by
\[
a \mapsto \begin{pmatrix}
a & 0 \\
0 & 1
\end{pmatrix}
\]
If there is no ambiguity, we simply write $\theta_A$ for this map. Given a unital $\ast$-homomorphism $\varphi : A\to B$, we write $\varphi_n$ for the induced maps in a variety of situations, such as $M_n(A)\to M_n(B), GL_n(A)\to GL_n(B), Lg_n(A)\to Lg_n(B)$, etc. Furthermore, when there is no ambiguity, we once again drop the subscript and denote the map by $\varphi$. Also, when dealing with modules over a C*-algebra, we will implicitly be referring to finitely generated \emph{right} modules. \\

Given two topological spaces $X$ and $Y$, we will write $[X,Y]$ for the set of free homotopy classes of continuous maps between them. If $X$ and $Y$ are pointed spaces, then we write $[X,Y]_{\ast}$ for the set of based homotopy classes of continuous functions based at those distinguished points. Here, we will be concerned with three pointed spaces associated to a unital C*-algebra $A$: $GL_n(A)$, as a subspace of $M_n(A)$ with base point $I_n$; $Lg_m(A)$, as a subspace of $A^m$ with base point $e_m$; and $Lc_m(A)$, as a subspace of $A^m$ with base point $e_m$.


\section{Homotopical Stable Ranks of Pullbacks}
Given unital $\ast$-homomorphisms $\gamma:C\to D$ and $\delta :B\to D$, we consider the pullback
\[
A := B\oplus_D C = \{(b,c)\in B\oplus C : \delta(b) = \gamma(c)\}
\]
As is customary, $A$ is best described by a pullback diagram, which we fix throughout the section
\begin{equation}\label{eqn:pullback}
\xymatrix{
A\ar[r]^{\alpha}\ar[d]^{\beta} &B\ar[d]^{\delta} \\
C\ar[r]^{\gamma} & D
}
\end{equation}
where $\alpha$ and $\beta$ are the projection maps. Furthermore, we assume that either $\gamma$ or $\delta$ is surjective. As pointed out in the example following \cite[Theorem 4.1]{brown}, this is quite a natural assumption when considering stable ranks. The goal then is to determine $gsr(A)$ and $csr(A)$ in terms of those of $B$ and $C$. To put things in perspective, we recall that the topological stable rank of $A$ may be estimated by \cite[Theorem 4.1]{brown}
\[
tsr(A) \leq \max\{tsr(B),tsr(C)\}
\]
The next example shows that the corresponding estimate for homotopical stable ranks cannot hold.
\begin{ex}\label{ex:sr_pullback}
Consider the pullback diagram
\[
\xymatrix{
C(\S^n)\ar[r]\ar[d] & C(\D^n)\ar[d]^{\delta} \\
C(\D^n)\ar[r]^{\gamma} & C(\S^{n-1})
}
\]
where $\gamma$ and $\delta$ are the natural restriction maps. Since $\D^n$ is contractible, $gsr(C(\D^n)) = 1$, but $gsr(C(\S^n)) > 1$ if $n\geq 5$ (See Example \ref{ex:gsr_suspension})
\end{ex}

\subsection{General Stable Rank}

To determine $gsr(A)$, we now describe a recipe due to Milnor \cite{milnor} to construct projective modules over $A$. Given a unital ring homomorphism $f:R\to S$ and a right $R$-module $M$, we write $f_{\#}(M)$ for the right $S$-module $S\otimes_R M$, and denote by $f_{\ast}$ the canonical map $M\to f_{\#}(M)$ given by $m \mapsto 1\otimes_R m$. Note that if $M$ is a free $R$-module with basis $\{b_{\alpha}\}$, then $f_{\#}(M)$ is a free $S$-module with basis $\{f_{\ast}(b_{\alpha})\}$. \\

Now consider a pullback diagram as above. Let $P$ and $Q$ be projective modules over $B$ and $C$ respectively, and suppose we are given a $D$-isomorphism
\[
h : \delta_{\#}(P) \to \gamma_{\#}(Q)
\]
Then consider
\[
M := \{(p,q) \in P\oplus Q : h\circ \delta_{\ast}(p) = \gamma_{\ast}(q)\}
\]
$M$ has a natural right $A$-module structure given by $(p,q)\cdot a := (p\cdot\alpha(a), q\cdot\beta(a))$. We denote this module by $M(P,Q,h)$. Milnor now proves the following 
\begin{thm}\cite[Theorem 2.1-2.3]{milnor}\label{thm:milnor_pullback}
\begin{enumerate}
\item The module $M = M(P,Q,h)$ is projective over $A$. Furthermore, if $P$ and $Q$ are finitely generated over $B$ and $C$ respectively, then $M$ is finitely generated over $A$.
\item Every projective $A$-module is isomorphic to $M(P,Q,h)$ for some suitably chosen $P,Q$ and $h$.
\item The modules $P$ and $Q$ are naturally isomorphic to $\alpha_{\#}(M)$ and $\beta_{\#}(M)$ respectively.
\end{enumerate}
\end{thm}
Furthermore, one has the following result. Recall that, for our purposes, we are only interested in C*-algebras that have the IBN property.
\begin{prop}\cite[Corollary 13.11]{magurn}\label{prop:double_coset_iso}
Given a pullback diagram as above, suppose $B$ or $C$ has the IBN property. Let $h_1,h_2\in GL_n(D)$, then $M(B^n,C^n,h_1)$ $\cong M(B^n,C^n,h_2)$ if and only if $h_1 = \delta(S_1)h_2\gamma(S_2)$ for some $S_1 \in GL_n(B), S_2\in GL_n(C)$.
\end{prop}

For $h_1,h_2\in GL_n(D)$, we write $h_1\sim h_2$ if and only if $\exists S_1\in GL_n(B), S_2\in GL_n(C)$ such that $h_1 = \delta(S_1)h_2\gamma(S_2)$. Note that this is an equivalence relation, whose equivalence classes are the double cosets
\[
\delta(GL_n(B))\backslash GL_n(D)/\gamma(GL_n(C))
\]
Given two elements $h_1, h_2 \in GL_n(D)$, we write $h_1\sim_h h_2$ if there is a path $f:\I \to GL_n(D)$ such that $f(0) = h_1$ and $f(1) = h_2$. The following observation now guides our investigation.

\begin{lemma}\label{lem:homotopic_pullback}
Consider the pullback diagram as above, and suppose $h_1, h_2\in GL_n(D)$ be such that $h_1\sim_h h_2$. Then $M(B^n,C^n,h_1)\cong M(B^n,C^n,h_2)$
\end{lemma}
\begin{proof}
Without loss of generality, assume that $\gamma$ is surjective. Since $h_2^{-1}h_1\in GL_n^0(D), \exists S_2\in GL_n(C)$ such that $h_2^{-1}h_1 = \gamma(S_2)$, and so $h_1 = \delta(I_{B^n})h_2\gamma(S_2)$. The result follows by Proposition \ref{prop:double_coset_iso}.
\end{proof}

Consider the sequence of groups
\[
\{1_D\} = GL_0(D) \hookrightarrow D^{\times} = GL_1(D)\hookrightarrow GL_2(D) \hookrightarrow \ldots
\]
For a compact Hausdorff space $X$, this induces a sequence of homotopy groups (of maps based at the identity)
\[
[X,GL_0(D)]_{\ast}\to [X,GL_1(D)]_{\ast}\to [X,GL_2(D)]_{\ast}\to \ldots
\]
We define
\begin{itemize}
\item $\inj_X(D)$ to be the least $n\geq 1$ such that the map $(\theta_D)_{\ast}: [X,GL_{m-1}(D)]_{\ast}$ $ \to [X,GL_m(D)]_{\ast}$ is injective for all $m\geq n$.
\item $\surj_X(D)$ to be the least $n\geq 1$ such that $(\theta_D)_{\ast}: [X,GL_{m-1}(D)]_{\ast}\to [X,GL_m(D)]_{\ast}$ is surjective for all $m\geq n$.
\item For $X = \S^n$, we write $\inj_n(D)$ and $\surj_n(D)$ for $\inj_X(D)$ and $\text{surj}_X(D)$ respectively.
\end{itemize}

\begin{rem}\label{rem:split_epi}
Let $A$ be a unital C*-algebra, and $X$ a compact Hausdorff space. Then, note that the natural map $Lg_m(C(X)\otimes A) \hookrightarrow C(X,Lg_m(A))$ given by evaluation is a homeomorphism (See \cite[Lemma 2.3]{rieffel2}). It follows that
\[
\pi_0(Lg_m(C(X)\otimes A)) = [X,Lg_m(A)]
\]
where the right hand side refers to free homotopy classes of maps from $X$ to $Lg_m(A)$. Furthermore, evaluation at a point gives a split epimorphism $C(X)\otimes A\to A$. Hence, it follows from Remark \ref{rem:properties}, (5) that
\begin{gather*}
csr(C(X)\otimes A)\geq csr(A),\text{ and } gsr(C(X)\otimes A)\geq gsr(A)
\end{gather*}
\end{rem}

Before we begin, note that $GL_n(A)$ is an open subset of a locally path connected space, so connected components in $GL_n(A)$ coincide with path components. The same is true for $Lg_m(A)$ and $Lc_m(A)$ (See \cite[Section 1]{rieffel2}) as well, and we will use this fact implicitly henceforth.

\begin{lemma}\label{lem:forgetful}
Let $X$ be a compact Hausdorff space, $A$ a unital C*-algebra, and $m\in \N$. Consider the forgetful map
\[
F: [X,Lg_m(A)]_{\ast} \to [X,Lg_m(A)]
\]
If $m\geq csr(A)$, then $F$ is surjective. If $m\geq gsr(\T A)$, then $F$ is injective.
\end{lemma}
\begin{proof}
Suppose $m\geq csr(A)$, then $Lc_m(A) = Lg_m(A)$ and $GL_m^0(A)$ acts transitively on $Lg_m(A)$. Let $x_0\in X$ be a fixed base point, and $f:X\to Lg_m(A)$ a continuous function. Then $\exists T\in GL_m^0(A)$ such that $T(f(x_0)) = e_m$. Let $g:X\to Lg_m(A)$ be given by $g(x) := T(f(x))$. If $h:\I \to GL_m(A)$ is a path such that $h(0) = I_n$ and $h(1) = T$, then the homotopy $H:\I\times X\to Lg_m(A)$ given by $H(t,x) := h(t)(f(x))$ is such that $H(0,x) = f(x)$, and $H(1,x) = g(x)$ for all $x\in X$. Hence, $[f] = [g]$ in $[X,Lg_m(A)]$. Since $g(x_0) = e_m$, we see that $F$ is surjective. \\

Now suppose $m \geq gsr(\T A)$, then $m\geq gsr(A)$, so $Lc_m(A) = Lg_m(A)$. Let $f$ and $g$ be two continuous functions from $X$ to $Lg_m(A)$ such that $f(x_0) = g(x_0) = e_m$, and suppose that there is a free homotopy $H:\I \times X\to Lg_m(A)$ such that
\[
H(0,x) = f(x),\text{ and } H(1,x) = g(x)
\]
for all $x\in X$. Consider $\gamma : \I \to Lg_m(A)$ by $\gamma(t) := H(t,x_0)$, then $\gamma(0) = \gamma(1) = e_m$, so $\gamma$ induces a map $\overline{\gamma}: \T\to Lg_m(A)$. Since $m\geq gsr(\T A)$, $Lc_m(\T A) = Lg_m(\T A)$, and the map $t: GL_m(\T A)\to Lc_m(\T A)$ is surjective. Identifying $Lg_m(\T A)$ with $C(\T, Lg_m(A))$ and $GL_m(\T A)$ with $C(\T, GL_m(A))$, we see that the map $t$ induces a surjective map $t: C(\T,GL_m(A))\to C(\T, Lg_m(A))$. Hence, $\exists h:\T\to GL_m(A)$ such that
\[
h(z)e_m = \overline{\gamma}(z)
\]
for all $z\in \T$. In particular, $h(1)e_m = e_m$ so that $h(1)^{-1}e_m = e_m$. Define $\overline{h} : \T \to GL_m(A)$ by $\overline{h}(z) := h(z)h(1)^{-1}$, then $\overline{h}(1) = I_n$ and $\overline{h}(z)e_m = \overline{\gamma}(z)$ for all $z\in \T$. This map $\overline{h}$ induces a map $\widetilde{h}:\I\to GL_m(A)$ such that $\widetilde{h}(0) = \widetilde{h}(1) = I_n$ and $\widetilde{h}(t)e_m = \gamma(t)$ for all $t\in \I$. Now define a homotopy $\widetilde{H} : \I\times X\to Lg_m(A)$ by
\[
\widetilde{H}(t,x) := \widetilde{h}(t)^{-1}(H(t,x))
\]
Then $\widetilde{H}(0,x) = f(x)$ and $\widetilde{H}(1,x) = g(x)$ for all $x\in X$. Furthermore, $\widetilde{H}(t,x_0)= e_m$ for all $t\in \I$, so $\widetilde{H}$ defines a base-point preserving homotopy from $f$ to $g$. Thus, the map $F : [X,Lg_m(A)]_{\ast} \to [X,Lg_m(A)]$ is injective.
\end{proof}

We will be most interested in the following quantities, which appear in the estimates for the homotopical stable ranks of a pullback.

\begin{prop}\label{prop:inj_surj_bounds}
For a unital C*-algebra $D$,
\begin{align*}
\surj_0(D) &\leq csr(D), \\
\inj_0(D)&\leq gsr(\T D), \text{ and }\\
\max\{\inj_0(D), \surj_1(D)\} &\leq csr(\T D) \\
\end{align*}
\end{prop}
\begin{proof}
Consider the long exact sequence of homotopy groups arising from the fibration $TL_m(D)\to GL_m(D)\to Lc_m(D)$ (See Equation \ref{eqn:les_fibration})
\[
\ldots \to \pi_{n+1}(Lc_m(D))\to \pi_n(GL_{m-1}(D))\to \pi_n(GL_m(D))\to \pi_n(Lc_m(D))\to \ldots
\]
which ends in a sequence of pointed sets $\pi_0(GL_{m-1}(D))\to \pi_0(GL_m(D))\to \pi_0(Lc_m(D))$. \\

For the first inequality, suppose $m\geq csr(D)$, then $m\geq gsr(D)$, so $Lc_m(D) = Lg_m(D)$ is connected. Hence, the map $\pi_0(GL_{m-1}(D))\to \pi_0(GL_m(D))$ is surjective. \\

For the second inequality, suppose $m\geq gsr(\T D)$, then the natural map $t: GL_m(\T D) \to Lc_m(\T D)$ is surjective. As mentioned in Lemma \ref{lem:forgetful}, it follows that, for any loop $\overline{\gamma}:\T \to Lg_m(A)$ based at $e_m$, there exists a loop $\overline{h} : \T\to GL_m(A)$ such that $\overline{h}(1) = I_m$ and $\overline{h}(z)e_m = \overline{\gamma}(z)$ for all $z\in \T$. Hence, the map
\[
\pi_1(GL_m(D))\to \pi_1(Lg_m(D))
\]
is surjective. By exactness of the above sequence, this implies that the map $\pi_0(GL_{m-1}(D))\to \pi_0(GL_m(D))$ is injective. \\

For the third inequality, if $m\geq csr(\T D)$, then $m\geq csr(D)$, so $Lc_m(D) = Lg_m(D)$ is connected. Furthermore, $Lg_m(\T D) = C(\T, Lg_m(D))$, so we see that
\[
0 = \pi_0(Lg_m(C(\T,D)) = \pi_0(C(\T, Lg_m(D)) = [\T, Lg_m(D)]
\]
By Remark \ref{rem:properties} (2), $m\geq gsr(\T D)$, so $\pi_1(Lg_m(D))$ is trivial by Lemma \ref{lem:forgetful}. The exactness of the above sequence now implies that the map $\pi_1(GL_{m-1}(D))\to \pi_1(GL_m(D))$ is surjective, and the map $\pi_0(GL_{m-1}(D))\to \pi_0(GL_m(D))$ is injective.
\end{proof}

We are now ready to prove an estimate for the general stable rank of the pullback as in Equation \ref{eqn:pullback}

\begin{thm}\label{thm:gsr_pullback}
Given a pullback diagram as above with either $\gamma$ or $\delta$ surjective,
\[
gsr(A)\leq \max\{csr(B), csr(C), \inj_0(D)\}
\]
Furthermore, if $K_1(D) = 0$, then
\[
gsr(A) \leq \max\{gsr(B),gsr(C), \inj_0(D)\}
\]
\end{thm}
\begin{proof}
Let $m\geq \max\{csr(B),csr(C),\inj_0(D)\}$, and $M$ be a projective $A$-module such that $M\oplus A\cong A^m$. Then write $M = M(P,Q,h)$ for some $P,Q$ and $h$ as in Theorem \ref{thm:milnor_pullback}. Then
\[
M(P\oplus B,Q\oplus C, h\oplus I_D)\cong A^m = M(B^m,C^m,I_{D^m})
\]
Hence,
\[
P\oplus B\cong \alpha_{\#}(A^m) \cong B^m
\]
Since $m\geq csr(B)\geq gsr(B)$, it follows that $P\cong B^{m-1}$. Similarly, $Q\cong C^{m-1}$. Hence, we may think of $h\in GL_{m-1}(D)$. Now consider the diagram
\[
\xymatrix{
GL_{m-1}(B)\ar[r]^{\delta_{m-1}}\ar[d]_{\theta_B} & GL_{m-1}(D)\ar[d]^{\theta_D} & GL_{m-1}(C)\ar[l]_{\gamma_{m-1}}\ar[d]^{\theta_C} \\
GL_m(B)\ar[r]^{\delta_m} & GL_m(D) & GL_m(C)\ar[l]_{\gamma_m}
}
\]
By Proposition \ref{prop:double_coset_iso}, $\theta_D(h)\sim I_m$, so $\exists b\in GL_m(B)$, and $c\in GL_m(C)$ such that $\theta_D(h) = \delta_m(b)\gamma_m(c)$. Since $m\geq csr(B)$, Proposition \ref{prop:inj_surj_bounds} implies that $m\geq \surj_0(B)$, so $\exists b'\in GL_{m-1}(B)$ such that $b\sim_h \theta_B(b')$. Hence,
\[
\delta_m(b)\sim_h \delta_m(\theta_B(b')) = \theta_D(\delta_{m-1}(b'))
\]
Similarly, $\exists c'\in GL_{m-1}(C)$ such that $\gamma_m(c)\sim_h \theta_D(\gamma_{m-1}(c'))$ so that
\[
\theta_D(h) \sim_h \theta_D(\delta_{m-1}(b')\gamma_{m-1}(c'))
\]
Since $m\geq \inj_0(D)$, $h\sim_h \delta_{m-1}(b')\gamma_{m-1}(c')$ and so by Lemma \ref{lem:homotopic_pullback} and Proposition \ref{prop:double_coset_iso}, we have
\[
M \cong M(B^{m-1},C^{m-1},h)\cong A^{m-1}
\]
Hence, $m\geq gsr(A)$ as required. \\

Now consider the special case where $K_1(D) = 0$: We follow the proof of the first part of the theorem until we obtain $h\in GL_{m-1}(D)$. Note that, up until this point, we only used the fact that $m\geq \max\{gsr(B),gsr(C)\}$. Since $m\geq \inj_0(D)$ and $K_1(D) = 0$, it follows that $h\sim_h I_{D^{m-1}}$, so that $M \cong M(B^{m-1},C^{m-1},h)\cong A^{m-1}$ by Lemma \ref{lem:homotopic_pullback}. We conclude that $m\geq gsr(A)$ as required.
\end{proof}

Note that Example \ref{ex:sr_pullback} shows that the term $\inj_0(D)$ on the right hand side cannot be dropped, and furthermore, that equality can hold in the above estimate, since $\inj_0(C(\S^4)) = gsr(C(\S^5)) = 4$.

\subsection{Connected Stable Rank}
Given a pullback diagram $A = B\oplus_D C$ as before, we wish to estimate $csr(A)$ in terms of $csr(B)$ and $csr(C)$. To begin with, we provide an alternate proof of the estimate of $gsr(A)$. This will help guide us to a proof for the corresponding estimate for $csr(A)$ as well.

\begin{lemma}\label{lem:gln_pullback}
Suppose $S_1\in GL_n(B), S_2\in GL_n(C)$ such that $\delta(S_1) = \gamma(S_2)$, then there exists a unique $T\in GL_n(A)$ such that $\alpha(T) = S_1, \beta(T) = S_2$. 
\end{lemma}
\begin{proof}
Since $M_n(\C)$ is nuclear, we obtain a pullback diagram
\[
\xymatrix{
M_n(A)\ar[r]^{\alpha}\ar[d]^{\beta} & M_n(B)\ar[d]^{\delta} \\
M_n(C)\ar[r]^{\gamma} & M_n(D)
}
\]
by \cite[Theorem 3.9]{pedersen}. Hence, there exists unique $T\in M_n(A)$ such that $\alpha(T) = S_1, \beta(T) = S_2$. We show that $T \in GL_n(A)$: To see this, let $S_1' \in GL_n(B)$ such that $S_1'S_1 = S_1S_1' = I_{B^n}$, and similarly, $S_2'\in GL_n(C)$ such that $S_2S_2' = S_2'S_2 = I_{C^n}$. Then
\[
\delta(S_1')\delta(S_1) = \delta(S_1')\gamma(S_2) = I_{D^n} = \gamma(S_2')\gamma(S_2)
\]
Hence, $\delta(S_1') = \gamma(S_2')$, so $\exists T'\in M_n(A)$ such that $\alpha(T') = S_1', \beta(T') = S_2'$. Now note that
\[
\alpha(T)\alpha(T') - \alpha(I_{A^n}) = S_1S_1'-I_{B^n} = 0
\]
so $TT'-I_{A^n} \in \ker(\alpha)$. Similarly, $TT' - I_{A^n}\in \ker(\beta)$. Since $\ker(\alpha)\cap\ker(\beta) = \{0\}$, it follows that $TT'= I_{A^n}$, and similarly, $T'T = I_{A^n}$. Hence, $T\in GL_n(A)$. 
\end{proof}

We now prove a fact that is probably well-known, but one that we have not found an exact reference for. Since it is crucial to our arguments, we include a proof here. Note that, for a non-unital C*-algebra $A$, $A^+$ refers to its unitization.

\begin{prop}\label{prop:path_lifting}
If $\gamma : C\to D$ is a unital, surjective $\ast$-homomorphism, then it has the path lifting property for invertibles: Given $n\in \N$ and a path $g:[0,1]\to GL_n(D)$ and $S \in GL_n(C)$ such that $g(1) = \gamma(S)$, there is a path $f:[0,1]\to GL_n(C)$ such that $f(1) = S$ and $\gamma\circ f = g$. 
\end{prop}
\begin{proof}
Consider the cone over $M_n(D)$, $\mathcal{C}M_n(D) = C_0[0,1)\otimes M_n(D)$, and observe that
\[
\mathcal{C}M_n(D)^+ = \{g : [0,1]\to M_n(D) : g(1) \in \C I_n\}
\]
Since $C_0[0,1)\otimes M_n(\C)$ is exact, the induced map $\gamma : \mathcal{C}M_n(C)^+ \to \mathcal{C}M_n(D)^+$ is surjective. Suppose $g:[0,1]\to GL_n(D)$ and $S\in GL_n(C)$ are such that $g(1) = \gamma(S)$. Replacing $g$ by $g(\cdot)\gamma(S^{-1})$ we may assume without loss of generality that $g(1) = I$. Then $H:\I\times \I \to GL_n(D)$ given by $H(s,t) = g(1-s(1-t))$ is a continuous map such that $H(0,t) = I, H(1,t) = g(t)$, and $H(s,1) = I$. Thus, $H$ defines a path in $GL(\mathcal{C}M_n(D)^+)$ such that $H(0) = I, H(1) = g$. Hence, $g\in GL^0_1(\mathcal{C}M_n(D)^+)$. So $\exists f\in GL(\mathcal{C}M_n(C)^+)$ such that $\gamma(f) = g$. Since $f(1) \in \C I_n$ and $\gamma$ is linear, it follows that $f(1) = I_n$. Thus, $f$ satisfies the required conditions. 
\end{proof}

By way of a warm-up for the estimate of the connected stable rank, we now provide a second proof of the estimate of the general stable rank from Theorem \ref{thm:gsr_pullback}.

\begin{thm}\label{thm:gsr_pullback_alt}
Given a pullback diagram as above with either $\gamma$ or $\delta$ surjective,
\[
gsr(A) \leq \max\{csr(B),csr(C),\inj_0(D)\}
\]
\end{thm}
\begin{proof}
Suppose $n\geq \max\{csr(B),csr(C),inj_0(D)\}$, then we wish to prove that $GL_n(A)$ acts transitively on $Lg_n(A)$. To this end, fix $v\in Lg_n(A)$, so that $\alpha(v) \in Lg_n(B)$. Since $n\geq csr(B),\exists S_1\in GL_n^0(B)$ such that $S_1\alpha(v) = e_n$. Hence, $\delta(S_1)w = e_n$ where $w = \delta(\alpha(v)) = \gamma(\beta(v))$. Similarly, $\exists S_2\in GL_n^0(C)$ such that $S_2\beta(v) = e_n$. Then $\gamma(S_2),\delta(S_1)\in GL_n^0(D)$, so $\gamma(S_2)\sim_h \delta(S_1)$. \\

Consider $S:= \delta(S_1)\gamma(S_2)^{-1}$, then $S\in GL_n^0(D)$ and $Se_n = e_n$. Hence, $S$ has the form
\[
S = \begin{pmatrix}
S' & 0 \\
c & 1
\end{pmatrix}
\]
where $c\in D^{n-1}$ and $S'\in GL_{n-1}(D)$. Since $S\sim_h I_{D^n}$, it follows that
\[
\begin{pmatrix}
S' & 0 \\
0 & 1
\end{pmatrix}\sim_h S \sim_h I_{D^n}
\]
where the first homotopy linearly sends $0$ to $c$. Since $n\geq inj_0(D), S' \sim_h I_{D^{n-1}}$ via a path $g : \I\to GL_{n-1}(D)$ such that $g(0) = S', g(1) = I$. By the path-lifting property of $\gamma$, there is a path $h : \I\to GL_{n-1}(C)$ such that $h(1) = I$ and $\gamma\circ h = g$. Let $c'\in C^{n-1}$ be such that $\gamma(c') = c$, and consider the element
\[
S_2' := \begin{pmatrix}
h(0) & 0\\
c' & 1
\end{pmatrix}S_2 \in GL_n^0(C)
\]
Then
\[
\gamma(S_2') = \begin{pmatrix}
\gamma(h(0)) & 0\\
c & 1 
\end{pmatrix} \gamma(S_2) = \begin{pmatrix}
S' & 0 \\
c & 1
\end{pmatrix}\gamma(S_2) = S\gamma(S_2) = \delta(S_1)
\]
By Lemma \ref{lem:gln_pullback}, $\exists T\in GL_n(A)$ such that $\alpha(T) = S_1, \beta(T) = S_2'$. Furthermore,
\[
S_2'(\beta(v)) = \begin{pmatrix}
h(0) & 0 \\
c' & 1
\end{pmatrix}S_2(\beta(v)) = \begin{pmatrix}
h(0) & 0 \\
c' & 1
\end{pmatrix}e_n = e_n
\]
Hence, $Tv - e_n \in A^n$ has the property that $\beta(Tv - e_n) = 0$, and $\alpha(Tv - e_n)=0$. Since $\ker(\alpha)\cap \ker(\beta) = \{0\}$, it follows that $Tv = e_n$, so $gsr(A) \leq n$.
\end{proof}

We now wish to determine conditions under which the operator $T$ produced in the above proof may be chosen to be in $GL_n^0(A)$. We begin with a lemma. Given a C*-algebra $C$ and $n\in \N$, we write $\iota : M_{n-1}(C)\to M_n(C)$ for the natural inclusion map, and define
\[
X_n(C) := \{f : [0,1]\to M_n(C) : f(0) \in \iota(M_{n-1}(C)), f(1) = 0\}
\]
For a C*-algebra $A$, we write $\mathcal{C} A$ for the cone $C_0[0,1)\otimes A$
\begin{lemma}
Let $\gamma : C\to D$ be a unital, surjective $\ast$-homomorphism, and $n\in \N$ be fixed. Then the induced map $\gamma : X_n(C)\to X_n(D)$ is also surjective.
\end{lemma}
\begin{proof}
Note that $X_n(C)$ is a pullback
\[
\xymatrix{
X_n(C)\ar[d]\ar[r] &  M_{n-1}(C)\ar[d]^{\iota}\\
\mathcal{C} M_n(C)\ar[r]^-{\rho} & M_n(C)
}
\]
where $\rho(f) = f(0)$. We now wish to appeal to \cite[Theorem 9.1]{pedersen}. To do this, consider the commuting diagram of short exact sequences
\[
\xymatrix{
0\ar[r] & \mathcal{C} \ker(\gamma_n)\ar[r]\ar[d]^{\overline{\rho}} & \mathcal{C} M_n(C)\ar[r]\ar[d]^{\rho} & \mathcal{C} M_n(D)\ar[r]\ar[d]^{\widetilde{\rho}} & 0 \\
0\ar[r] & \ker(\gamma_n)\ar[r] & M_n(C)\ar[r] & M_n(D)\ar[r] & 0 \\
0\ar[r] & \ker(\gamma_{n-1})\ar[u]_{\overline{\iota}}\ar[r] & M_{n-1}(C)\ar[u]_{\iota}\ar[r] & M_{n-1}(D)\ar[r]\ar[u]_{\widetilde{\iota}} & 0
}
\]
where $\gamma_k : M_k(C) \to M_k(D)$ is the map induced by $\gamma$, and the vertical maps are induced by $\rho$ and $\iota$. In order to conclude that the map $\gamma : X_n(C)\to X_n(D)$ is surjective, we must verify that $E = F$ where
\begin{align*}
E &:= \rho(\mathcal{C} M_n(C))\cap \iota(M_{n-1}(C))\cap \ker(\gamma_n), \text{ and } \\
F &:= \overline{\rho}(\mathcal{C} \ker(\gamma_n))\cap \iota(M_{n-1}(C)) + \rho(\mathcal{C} M_n(C))\cap \overline{\iota}(\ker(\gamma_{n-1}))
\end{align*}
Note that $\iota(M_{n-1}(C))\cap \ker(\gamma_n) = \overline{\iota}(\ker(\gamma_{n-1}))$, so
\[
E = \rho(\mathcal{C} M_n(C))\cap \iota(M_{n-1}(C))\cap \ker(\gamma_n) = \rho(\mathcal{C} M_n(C))\cap \overline{\iota}(\ker(\gamma_{n-1})) \subset F
\]
Furthermore, $\overline{\rho}$ is the restriction of $\rho$ to $\mathcal{C}\ker(\gamma_n)$, so $\overline{\rho}(\mathcal{C}\ker(\gamma_n)) \subset \rho(\mathcal{C}M_n(C))\cap \ker(\gamma_n)$. Hence,
\[
\overline{\rho}(\mathcal{C} \ker(\gamma_n))\cap \iota(M_{n-1}(C)) \subset \rho(\mathcal{C}M_n(C))\cap \ker(\gamma_n)\cap \iota(M_{n-1}(C)) = E
\]
and $F\subset E$ also holds. Thus, \cite[Theorem 9.1]{pedersen} applies, and $\gamma : X_n(C)\to X_n(D)$ is surjective.
\end{proof}

We conclude that the induced map $\gamma : X_n(C)^+ \to X_n(D)^+$ is also surjective, and note that
\[
X_n(C)^+ = \{f : [0,1]\to M_n(D) : \exists \lambda \in \C \text{ such that } f(0) \in \iota(M_{n-1}(D)) + \lambda I_n, f(1) = \lambda I_n\}
\]
In what follows, note that $\theta^n_C : GL_{n-1}(C)\to GL_n(C)$ denotes the natural inclusion of groups.
\begin{lemma}\label{lem:loop_lifting}
Let $\gamma : C\to D$ be a unital, surjective $\ast$-homomorphism, and suppose $n\in \N$ is such that $(\theta^n_D)_{\ast} : \pi_1(GL_{n-1}(D))\to \pi_1(GL_n(D))$ is surjective. If $g:\I \to GL_n(D)$ is a path such that $g(0) = g(1) = I_n$, then $\exists h:\I \to GL_n(C)$ such that $h(0) \in \theta^n_C(GL_{n-1}(C)), h(1) = I_n$, and $\gamma\circ h = g$
\end{lemma}
\begin{proof}
Consider $\overline{g} : \T \to GL_n(D)$ induced by $g$. Since $(\theta^n_D)_{\ast} : \pi_1(GL_{n-1}(D))\to \pi_1(GL_n(D))$ is surjective, $\exists f : \T \to GL_{n-1}(D)$ such that $f(1) = I_n$, and a homotopy $H : \I\times \T \to GL_n(D)$ such that $H(t,1) = I_n$ for all $t\in \I$, and
\[
H(0,z) = \overline{g}(z), \text{ and }  H(1,z) = (\theta^n_D\circ f)(z)
\]
for all $z\in \T$. Think of $f$ as a path $f: \I\to GL_{n-1}(D)$ such that $f(0) = f(1) = I_n$. By the path lifting property, $\exists \overline{f} : \I\to GL_{n-1}(C)$ such that $\overline{f}(1) = I_n$ and $\gamma\circ \overline{f} = f$.  \\

Define $\widetilde{f} : \I\to GL_n(C)$ by $\widetilde{f} := \theta^n_C\circ \overline{f}$, so that $\gamma\circ \widetilde{f} = \theta^n_D\circ f$. We may think of $H$ as a map $H:\I\times \I \to GL_n(D)$ such that $H(t,0) = H(t,1) = I_n$. Hence, $H$ defines a path
\[
H : \I \to GL(X_n(D)^+)
\]
such that $H(0) = g$ and $H(1) = \gamma\circ \widetilde{f}$. By the previous lemma and the path lifting property of Proposition \ref{prop:path_lifting}, $\exists \overline{H} : \I \to GL(X_n(C)^+)$ such that $\gamma\circ \overline{H} = H$ and $\overline{H}(1) = \widetilde{f}$. Now $h := \overline{H}(0) \in GL(X_n(C)^+)$ is a path $h:\I\to GL_n(C)$ such that
\[
\gamma\circ h = \gamma\circ \overline{H}(0) = H(0) = g
\]
Since $h(1) \in \C I_n$ and $\gamma$ is linear, it follows that $h(1) = g(1) = I_n$. Hence, $h(0) \in (\iota(M_{n-1}(C))+I_n)\cap GL_n(C)$, which implies that $h(0) \in \theta^n_C(GL_{n-1}(C))$.
\end{proof}

We are now in a position to prove the estimate on the connected stable rank of $A$ defined as a pullback as in Equation \ref{eqn:pullback}. Together with Proposition \ref{prop:inj_surj_bounds} and Theorem \ref{thm:gsr_pullback}, this completes the proof of Theorem \ref{thm:main_pullback}.

\begin{thm}\label{thm:csr_pullback}
Given a pullback diagram as above with either $\gamma$ or $\delta$ surjective,
\[
csr(A) \leq \max\{csr(B),csr(C),\inj_0(D),\surj_1(D)\}
\]
\end{thm}
\begin{proof}
Let $n\geq \max\{csr(B),csr(C),\inj_0(D),\surj_1(D)\}$, then we wish to prove that $GL_n^0(A)$ acts transitively on $Lg_n(A)$. To this end, fix $v\in Lg_n(A)$, so that $\alpha(v) \in Lg_n(B)$ and $\beta(v) \in Lg_n(C)$. By the proof of Theorem \ref{thm:gsr_pullback_alt}, $\exists S_1 \in GL_n^0(B)$, and $S_2\in GL_n^0(C)$ such that $S_1\alpha(v) = e_n, S_2\beta(v) = e_n$, and $\delta(S_1) = \gamma(S_2)$. Hence, we obtained $T \in GL_n(A)$ such that $\alpha(T) = S_1, \beta(T) = S_2$, and $T(v) =e_n$. \\

Now fix a path $\overline{S_1} : \I\to GL_n(B)$ such that $\overline{S_1}(0) = S_1$ and $\overline{S_1}(1) = I_n$, and a path $\overline{S_2} : \I\to GL_n(C)$ such that $\overline{S_2}(0) = S_2$, and $\overline{S_2}(1) = I_n$. Now consider $g:\I\to GL_n(D)$ given by
\[
g(t) = \delta(\overline{S_1}(t))\gamma(\overline{S_2}(t))^{-1}
\]
Then $g(0) = g(1) = I_n$. So by the previous lemma, $\exists h: \I\to GL_n(C)$ such that $h(0) \in \theta^n_C(GL_{n-1}(C)), h(1) = I_n$ and $\gamma\circ h = g$. Now define
\[
S_2' := h(0)S_2
\]
Then $S_2'\beta(v) = e_n$ because $h(0)e_n = e_n$, and $\gamma(S_2') = g(0)\gamma(S_2) = \delta(S_1)$. Hence, by Lemma \ref{lem:gln_pullback}, $\exists T' \in GL_n(A)$ such that $\alpha(T') = S_1, \beta(T') = S_2'$, and $T'(v) = e_n$. We wish to show that $T' \in GL_n^0(A)$. \\

Define $\overline{S_2'} :\I \to GL_n(C)$ by $\overline{S_2'}(t) := h(t)\overline{S_2}(t)$. Since $\gamma\circ h = g$, we have
\[
\gamma\circ \overline{S_2'} = \delta\circ \overline{S_1}
\]
Since $C[0,1]\otimes M_n(\C)$ is nuclear, by \cite[Theorem 3.9]{pedersen} we have a pullback
\[
\xymatrix{
C[0,1]\otimes M_n(A) \ar[r]^{\alpha}\ar[d]^{\beta} & C[0,1]\otimes M_n(B)\ar[d]^{\delta} \\
C[0,1]\otimes M_n(C) \ar[r]^{\gamma} & C[0,1]\otimes M_n(D)
}
\]
so we obtain a path $f: \I\to M_n(A)$ such that $\alpha\circ f = \overline{S_1}$, and $\beta\circ f = \overline{S_2'}$. For each $t\in \I, \beta(f(t)) \in GL_n(C)$ and $\alpha(f(t))\in GL_n(B)$, so $f(t) \in GL_n(A)$ by Lemma \ref{lem:gln_pullback}. Hence, $f$ defines a path in $GL_n(A)$. Furthermore, $\alpha(f(0)) = S_1$ and $\beta(f(0)) = S_2'$, so by the uniqueness in Lemma \ref{lem:gln_pullback}, $f(0) = T'$. Similarly, $f(1) = I_n$, so $T' \in GL_n^0(A)$. Hence, $GL_n^0(A)$ acts transitively on $Lg_n(A)$, whence $csr(A)\leq n$.
\end{proof}


\section{Tensor products by commutative C*-algebras}

We now wish to calculate the homotopical stable ranks for algebras of the form $C(X)\otimes A$. Once again, we consider the general and connected stable ranks separately. 

\subsection{General Stable Rank}

To compute $gsr(C(X)\otimes A)$, we wish to describe all projective modules over $C(X)\otimes A$. If $A = \C$, by the Serre-Swan theorem, this amounts to describing all vector bundles over $X$. This is prohibitively difficult, of course, so we consider the potentially simpler situation, when $X$ is itself a suspension. \\

Let $X$ be a compact Hausdorff space, and $x_0\in X$ be a fixed base point. The reduced suspension of $X$ is $\Sigma X = (X\times \I)/\sim$ where
\[
(0,x)\sim (0,x_0), (1,x)\sim (1,x_0), \text{ and } (s,x_0)\sim (0,x_0)\quad\forall x\in X,s\in \I
\]
Now we observe that vector bundles of rank $n$ over $\Sigma X$ correspond to homotopy classes of maps from $X$ into $GL_n(\C)$ based at the identity. Hence, $gsr(C(\Sigma X))$ is the least $n\geq 1$ such that the map $[X,GL_{m-1}(\C)]_{\ast}\to [X,GL_m(\C)]_{\ast}$ induced by $\theta_{\C}$ is injective for all $m\geq n$. In our notation, this simply gives
\begin{equation}\label{eqn:gsr_sx}
gsr(C(\Sigma X)) = \inj_X(\C)
\end{equation}
This is precisely the observation used by Nica to give the first non-trivial calculation of the general stable rank.
\begin{ex}\cite[Proposition 5.5]{nica2}\label{ex:gsr_suspension}
\[
gsr(C(\S^d)) = \begin{cases}
\lceil \frac{d}{2}\rceil + 1 &: \text{ if } d>4, \text{ and } d\notin 4\Z \\
\lceil \frac{d}{2}\rceil &:\text{ if } d>4, \text{ and } d\in 4\Z \\
1 &: d\leq 4
\end{cases}
\]
\end{ex}
The goal of this section is to expand on this idea, by describing projective modules over $C(\Sigma X)\otimes A$, which allows us to prove an analogue of Equation \ref{eqn:gsr_sx}. To begin with, we fix a unital C*-algebra $A$, and we identify functions $f:\Sigma X\to A$ with functions $f:\I\times X\to A$ such that
\begin{equation}\label{eqn:functions_suspension}
f(0,x) = f(1,x) = f(s,x_0) \quad\forall x\in X,s\in \I
\end{equation}
We now follow the work of Rieffel \cite{rieffel3} to construct projective modules over $C(\Sigma X)\otimes A$. Given a projective right $A$-module $V$, $\text{Aut}_A(V)$ is equipped with the point-norm topology, and has the base point $\text{id}_V$. Let $C_{x_0}(X,\text{Aut}_A(V))$ be the space of continuous functions from $u: X\to \text{Aut}_A(V)$ such that $u(x_0) = \text{id}_V$. Given a projective right $A$-module $V$ and $u\in C_{x_0}(X,\text{Aut}_A(V))$, we define
\begin{gather*}
M(u) = \{\varphi : \I\times X\to V : \varphi(0,x) = \varphi(s,x_0) \text{ and } \\
\varphi(1,x) = u(x)\varphi(0,x)\quad\forall x\in X, s\in \I\}
\end{gather*}
Note that $M(u)$ is a right $C(\Sigma X)\otimes A$-module with the action given by
\[
(\varphi\cdot f)(t,x) := \varphi(t,x)f(t,x)
\]
\begin{lemma}\label{lem:path_connected}
If $u_0, u_1$ are path connected in $C_{x_0}(X,\text{Aut}_A(V))$, then $M(u_0)\cong M(u_1)$
\end{lemma}
\begin{proof}
Let $H:\I\to C_{x_0}(X,\text{Aut}_A(V))$ be a path such that $H(0) = u_0, H(1) = u_1$, then we may think of $H$ as a map $H:\I\times X\to \text{Aut}_A(V)$. Define $F: M(u_0)\to M(u_1)$ by
\[
F(\varphi)(s,x) := H(s,x)u_0(x)^{-1}\varphi(s,x)
\]
so $F(\varphi)(1,x) = H(1,x)u_0(x)^{-1}\varphi(1,x) = u_1(x)\varphi(0,x)$ and $F(\varphi)(0,x) = \varphi(0,x)$. Hence, $F$ is well-defined. Also, $F$ is clearly a module homomorphism because the action of $C(\Sigma X)\otimes A$ is on the right. To show that $F$ is an isomorphism, we define $G:M(u_1)\to M(u_0)$ by
\[
G(\psi)(s,x) := u_0(x)H(s,x)^{-1}\psi(s,x)
\]
Then $G$ is a well-defined module homomorphism such that $G\circ F = \text{id}_{M(u_0)}$ and $F\circ G = \text{id}_{M(u_1)}$
\end{proof}

Given two projective right $A$-modules $V_1$ and $V_2$, $u_1\in C_{x_0}(X,\text{Aut}_A(V_1))$, and $u_2\in C_{x_0}(X,\text{Aut}_A(V_2))$, we write $u_1\oplus u_2 \in C_{x_0}(X,\text{Aut}_A(V_1\oplus V_2))$ for the map
\[
(u_1\oplus u_2)(x)(v_1,v_2) := (u_1(x)(v_1), u_2(x)(v_2))
\]
The proof of the next two lemmas is entirely obvious from the definition.
\begin{lemma}\label{lem:direct_sum}
If $V_1$ and $V_2$ are projective, right $A$-modules and $u_1\in C_{x_0}(X,\text{Aut}_A(V_1))$, $u_2\in C_{x_0}(X,\text{Aut}_A(V_2))$, then
\[
M(u_1\oplus u_2)\cong M(u_1)\oplus M(u_2)
\]
\end{lemma}

\begin{lemma}\label{lem:trivial_module}
If $\iota_{A^n}\in C_{x_0}(X,GL_n(A))$ denotes the identity automorphism on $A^n$, then
\[
M(\iota_{A^n}) \cong (C(\Sigma X)\otimes A)^n
\]
\end{lemma}
\begin{lemma}\label{lem:iso_connected}
Let $u,v\in C_{x_0}(X,\text{Aut}_A(V))$ such that $M(u)\cong M(v)$, then $\exists w\in C(X,\text{Aut}_A(V))$ such that $v$ is path connected to $wuw^{-1}$ in $C_{x_0}(X,\text{Aut}_A(V))$
\end{lemma}
\begin{proof}
Note that $M(u)$ and $M(v)$ are both section algebras of locally trivial bundles over $\Sigma X$ with fibers $V$, so if $M(u)\cong M(v)$, then the isomorphism is implemented by a map $\widehat{g}:\Sigma X\to \text{Aut}_A(V)$. As in Equation \ref{eqn:functions_suspension}, we identify $\widehat{g}$ with a function $g:\I\times X\to \text{Aut}_A(V)$ such that
\[
g(0,x) =g(1,x) = g(s,x_0) = \text{id}_V \quad\forall x\in X,\quad\forall s\in \I
\]
Then note that for any $\varphi \in M(u)$,
\[
v(x)g(0,x)\varphi(0,x) = v(x)(g(\varphi))(0,x) = g(\varphi)(1,x) = g(1,x)\varphi(1,x) = g(1,x)u(x)\varphi(0,x)
\]
Hence, it follows that $v(x)g(0,x) = g(1,x)u(x)$ for all $x\in X$, so that
\[
v(x) = g(1,x)u(x)g(0,x)^{-1}\quad\forall x\in X
\]
Let $w:X\to \text{Aut}_A(V)$ be given by $w(x) := g(0,x)$, and let $H: \I\times X\to C_{x_0}(X,\text{Aut}_A(V))$ be given by $H(t,x) := g(t,x)u(x)g(0,x)^{-1}$. Then
\[
H(0,x) = w(x)u(x)w(x)^{-1}\qquad\text{ and } \qquad H(1,x) = v(x)
\]
Furthermore, $H(s,x_0) = g(s,x_0)u(x_0)g(0,x_0)^{-1}=u(x_0) = \text{id}_V$ for all $s\in \I$. Hence $H$ implements a homotopy $v\sim_h wuw^{-1}$ in $C_{x_0}(X,\text{Aut}_A(V))$.
\end{proof}

\begin{lemma}\label{lem:module_csxa}
Every projective $C(\Sigma X)\otimes A$-module is isomorphic to $M(u)$ for some projective $A$-module $V$ and some $u \in C_{x_0}(X,\text{Aut}_A(V))$.
\end{lemma}
\begin{proof}
A projective $C(\Sigma X)\otimes A$-module $M$ is isomorphic to $P((C(\Sigma X)\otimes A)^n)$ for some projection $P\in M_n(C(\Sigma X)\otimes A)$. We identify $P$ with a map $P:\I\times X\to M_n(A)$ satisfying Equation \ref{eqn:functions_suspension}. Let $p := P(0,x_0)$ and $V := p(A^n)$. If we think of $P$ as a path $P:\I\to C(X)\otimes M_n(A)$, then there is a path of unitaries $U: \I\to GL(C(X)\otimes M_n(A))$ such that
\[
P(t,x) = U(t,x)^{-1}P(0,x)U(t,x) = U(t,x)^{-1}pU(t,x)
\]
Furthermore, we have $U(0,x) = \text{id}_{A^n} = U(s,x_0) \quad\forall x\in X$, and $s\in \I$. Hence,
\[
U(1,x)^{-1}pU(1,x) = U(1,x)^{-1}P(0,x)U(1,x) = P(1,x) = P(0,x) = p
\]
so $U(1,x)p = pU(1,x)$, so we define $u(x) := U(1,x)p \in \text{Aut}_A(V)$ and $u\in C_{x_0}(X,\text{Aut}_A(V))$. Finally, if $f\in P(C(\Sigma X)\otimes A)^n)$, then we think of $f$ as a function $f:\I\times X\to A^n$ satisfying Equation \ref{eqn:functions_suspension} and $P(t,x)f(t,x) = f(t,x)$. Hence, we may define $\varphi : \I\times X\to V$ by
\[
\varphi(t,x) := U(t,x)f(t,x)
\]
and this is well-defined because
\[
p\varphi(t,x) = P(0,x)U(t,x)f(t,x) = U(t,x)P(t,x)f(t,x) = U(t,x)f(t,x) = \varphi(t,x)
\]
Furthermore, $\varphi(1,x) = U(1,x)f(1,x) = U(1,x)f(0,x)$ and $\varphi(0,x) = f(0,x)$. Hence, $\varphi \in M(u)$. It is then easy to check that the map that sends $f$ to $\varphi$ is an isomorphism from $P(C(\Sigma X)\otimes A)^n)$ to $M(u)$.
\end{proof}

We are now ready to prove the main theorem of this section. Recall that $\inj_X(A)$ is the least $n\geq 1$ such that the map $(\theta_A)_{\ast} : [X,GL_{m-1}(A)]_{\ast}\to [X,GL_m(A)]_{\ast}$ is injective for all $m\geq n$.
\begin{thm}\label{thm:csxa}
\[
gsr(C(\Sigma X)\otimes A) = \max\{gsr(A),\inj_X(A)\}
\]
\end{thm}
\begin{proof}
For simplicity of notation, write $B:= C(\Sigma X)\otimes A$. Let $n \geq\max\{gsr(A),$ $\inj_X(A)\}$, and let $P$ be a projective module over $B$ such that $P\oplus B\cong B^n$. By Lemma \ref{lem:module_csxa}, $\exists$ a projective $A$-module $V$ and a map $u\in C_{x_0}(X,\text{Aut}_A(V))$ such that $P\cong M(u)$. The map $\pi : B\to A$ given by evaluation at $[(0,x_0)]\in \Sigma X$ is a ring homomorphism, so
\[
\pi_{\#}(P)\oplus A\cong A^n
\]
But $\pi_{\#}(P) \cong V$ and $gsr(A)\leq n$ so $V\cong A^{n-1}$. Hence, we may think of $u\in C_{x_0}(X,GL_{n-1}(A))$. Now note that
\[
M(u\oplus \iota_A)\cong M(\iota_{A^n})
\]
so by Lemma \ref{lem:iso_connected} and Lemma \ref{lem:trivial_module}, $u\oplus \iota_A \sim_h \iota_{A^n}$ in $C_{x_0}(X,GL_n(A))$. Since $n\geq \inj_X(A)$, it follows that $u\sim_h \iota_{A^{n-1}}$ in $C_{x_0}(X,GL_{n-1}(A))$, whence $P\cong B^{n-1}$ by Lemma \ref{lem:path_connected}. Hence, $gsr(B)\leq n$ as required. \\

For the reverse inequality, let $n \geq gsr(B)$, then by Remark \ref{rem:split_epi}, $n\geq gsr(A)$. Now suppose $u\in C_{x_0}(X,GL_{n-1}(A))$ is such that $u\oplus \iota_A \sim_h \iota_{A^n}$ in $C_{x_0}(X,GL_n(A))$, then let $P = M(u)$. By Lemmas \ref{lem:direct_sum} and \ref{lem:trivial_module},
\[
P\oplus B\cong M(u\oplus \iota_A) \cong M(\iota_{A^n}) \cong B^n
\]
By hypothesis, $P\cong B^{n-1} = M(\iota_{A^{n-1}})$. By Lemma \ref{lem:iso_connected}, $u\sim_h \iota_{A^{n-1}}$ in $C_{x_0}(X,GL_{n-1}(A))$, and so $\inj_X(A)\leq n$, completing the proof.
\end{proof}

\subsection{Connected Stable Rank}\label{subsection:csr_commutative}

Let $A$ be a C*-algebra, and $X$ a compact Hausdorff space, then we wish to determine estimates for $csr(C(X)\otimes A)$ in terms of $\dim(X)$ and other parameters that depend on $A$. To this end, we notice that if $X$ is a CW-complex of dimension atmost $n$, we may write $X = X_0\cup _{\varphi}\D^n$, where $X_0$ is a CW-complex of dimension atmost $n$, and $\varphi : \S^n\to X_0$ is the attaching map. By \cite[Lemma 1.4]{nagisa}, we have a pullback diagram
\[
\xymatrix{
C(X)\otimes A\ar[r]\ar[d] & C(X_0)\otimes A\ar[d]^{\varphi^{\ast}} \\
C(\D^n)\otimes A\ar[r]^{\gamma} & C(\S^{n-1})\otimes A
}
\]
where $\gamma$ is the restriction map. By Theorem \ref{thm:csr_pullback}, we get
\begin{align*}
csr(C(X)\otimes A)\leq & \max\{csr(C(X_0)\otimes A), csr(C(\D^n)\otimes A), \\ & \inj_0(C(\S^{n-1})\otimes A), \surj_1(C(\S^{n-1})\otimes A)\}
\end{align*}
In order to estimate $\text{inj}_0(C(\S^{n-1})\otimes A)$, we observe that
\begin{lemma}\label{lem:inj_cx}
If $X$ is a compact Hausdorff space and $A$ a unital C*-algebra, then
\[
\inj_0(C(X)\otimes A) = \inj_X(A)
\]
\end{lemma}
\begin{proof}
First note that $GL_k(C(X)\otimes A) = C(X,GL_k(A))$, so $[X,GL_k(A)] = \pi_0(GL_k(C(X)\otimes A)$. Hence, $\inj_0(C(X)\otimes A)$ is the least $m\geq 1$ such that
\[
(\theta_A)_{\ast}: [X,GL_{k-1}(A)]\to [X,GL_k(A)]
\]
is injective for all $k\geq m$. Let $x_0\in X$ be a fixed base point and $C_{x_0}(X,GL_k(A))$ be the set of all continous maps $f:X\to GL_k(A)$ such that $f(x_0) = I_n$. Now suppose $n\geq \inj_0(C(X)\otimes A), k\geq n$ and $f,g\in C_{x_0}(X,GL_{k-1}(A))$ are such that $\theta_A(f)\sim_h \theta_A(g)$ in $C(X,GL_k(A))$. Since $n\geq \inj_0(C(X)\otimes A)$, the above comments imply that there is a free homotopy $H:\I\times X\to GL_{k-1}(A)$ such that $H(0,x) = f(x)$ and $H(1,x) = g(x)$ for all $x\in X$. Then, $\widetilde{H}(t,x) := H(t,x)H(t,x_0)^{-1}$ defines a base point preserving homotopy from $f$ to $g$. Hence, $n\geq \inj_X(A)$. \\

Conversely, let $n\geq \inj_X(A), k\geq n$, and let $f,g : X\to GL_{k-1}(A)$ be two maps such that $\theta_A(f)\sim_h \theta_A(g)$ in $C(X,GL_k(A))$. Then there is a homotopy $H:\I\times X\to GL_k(A)$ connecting $\theta_A(f)$ to $\theta_A(g)$. Hence, $\widetilde{H}$, as defined above, is a base-point preserving homotopy from $\theta_A(\alpha(f))$ to $\theta_A(\alpha(g))$ where $\alpha(h)(x) := h(x)h(x_0)^{-1}$. Since $n\geq \inj_X(A)$, $\alpha(f)$ is homotopic to $\alpha(g)$ in $C_{x_0}(X,GL_{k-1}(A))$. If $G:\I\times X\to GL_{k-1}(A)$ is a base-point preserving homotopy such that $G(0,x) = \alpha(f)(x)$ and $G(1,x) = \alpha(g)(x)$ for all $x\in X$, then $\widehat{G}(t,x) := G(t,x)G(t,x_0)$ is a homotopy connecting $f$ to $g$ in $C(X,GL_{k-1}(A))$. Hence, $n\geq \inj_0(C(X)\otimes A)$. This completes the proof.
\end{proof}
In order to estimate the term $\surj_1(C(\S^{n-1})\otimes A)$ that occurs in the above inequality, we turn to the work of Thomsen \cite{thomsen}, where he defines an axiomatic homology theory that will be relevant to us. Recall \cite{schochet} that a homology theory is a sequence $\{h_n\}$ of covariant functors from an admissible category $\mathcal{D}$ of C*-algebras to abelian groups which satisfies the following axioms:
\begin{itemize}
\item Homotopy Axiom: If $\varphi_0, \varphi_1 : A\to B$ are homotopic morphisms (in the sense described in the discussion preceding Theorem \ref{thm:homotopy_dom}), then $(\varphi_0)_{\ast} = (\varphi_1)_{\ast} : h_n(A)\to h_n(B)$ for all $n\in \N$.
\item Exactness axiom: Let $0\to J\to A\to B\to 0$ be a short exact sequence in $\mathcal{D}$, then there is a map $\partial : h_n(B)\to h_{n-1}(J)$ and a long exact sequence $\ldots \to h_n(J)\to h_n(A)\to h_n(B)\xrightarrow{\partial} h_{n-1}(J)\to h_{n-1}(A)\to \ldots$. The map $\partial$ is natural with respect to morphisms of short exact sequences.
\end{itemize}
These two axioms imply that any homology theory is additive: If $0\to J\to A\to B$ is a split exact sequence in $\mathcal{D}$, then there is a natural isomorphism $h_n(A)\cong h_n(J)\oplus h_n(B)$ for all $n\in \N$. \\

Now, let $A$ be a C*-algebra (not necessarily unital), $A^+$ the C*-algebra obtained by adjoining a unit to $A$, and consider $A$ as an ideal of $A^+$. For $m\in \N$, define 
\[
GL_m(A) := \{x\in GL_m(A^+) : x - I_{(A^+)^m} \in M_m(A)\}
\]
then Thomsen proves \cite[Theorem 2.5]{thomsen} that, for a fixed $m\in \N$, the functor
\[
h_n(A) := \pi_{n+1}(GL_m(A))
\]
defines a homology theory from the category of C*-algebras to the category of abelian groups.

\begin{lemma}\label{lem:surj_restriction}
Let $D = C(\S^{n-1})\otimes A$, then
\[
\max\{\inj_0(D), \surj_1(D)\} \leq \max\{\surj_1(A), \surj_n(A), \inj_{n-1}(A)\}
\]
\end{lemma}
\begin{proof}
Let $m\geq \max\{\surj_1(A), \surj_n(A), \inj_{n-1}(A)\}$. It follows from Lemma \ref{lem:inj_cx} that $\inj_{n-1}(A) = \inj_0(D)$, so $m\geq \inj_0(D)$. We have a split exact sequence
\[
0\to C_0(\R^{n-1})\otimes A\to C(\S^{n-1})\otimes A \to A\to 0
\]
Furthermore, by \cite[Lemma 2.3]{thomsen}, 
\[
\pi_n(GL_k(A)) \cong \pi_1(GL_k(C_0(\R^{n-1})\otimes A))
\]
By additivity of the functor $A\mapsto \pi_1(GL_k(A))$, there is a natural isomorphism
\[
\pi_1(GL_k(C(\S^{n-1}\otimes A))\cong \pi_n(GL_k(A))\oplus \pi_1(GL_k(A))
\]
for $k\in \{m,m-1\}$. Since $m\geq \max\{\surj_1(A), \surj_n(A)\}$, $m\geq \surj_1(D)$ as well.
\end{proof}

\begin{thm}\label{thm:csr_finite_dim}
Let $X$ be a compact Hausdorff space of dimension atmost $n$, then
\[
csr(C(X)\otimes A)\leq \max\{csr(A), \surj_k(A), \inj_{k-1}(A) : 1\leq k\leq n\}
\]
\end{thm}
\begin{proof}
If $X$ is a compact Hausdorff space of dimension atmost $n$, then $X$ is an inverse limit of compact metric spaces $(X_i)$ such that $\dim(X_i)\leq n$ \cite{mardesic}. Since $C(X)\otimes A\cong \lim C(X_i)\otimes A$, it follows from Remark \ref{rem:properties}, (7) that $csr(C(X)\otimes A)\leq \liminf csr(C(X_i)\otimes A)$. Furthermore, if $X$ is a compact metric space of dimension atmost $n$, then $X$ is an inverse limit of finite CW-complexes $(Y_i)$, such that $\dim(Y_i)\leq n$ \cite{freudenthal}. Once again, $csr(C(X)\otimes A))\leq \liminf C(Y_i)\otimes A$. Hence, it suffices to assume that $X$ is itself a finite CW-complex with $\dim(X)\leq n$. \\

By induction, we may assume that $X = X_0\cup_{\varphi} \D^n$ where $X_0$ is a finite CW-complex of dimension atmost $(n-1)$ and $\varphi : \S^{n-1}\to X_0$ is a continuous function. As mentioned at the start of this section, it follows that
\begin{align*}
csr(C(X)\otimes A)\leq & \max\{csr(C(X_0)\otimes A), csr(C(\D^n)\otimes A), \\ & \inj_0(C(\S^{n-1})\otimes A), \surj_1(C(\S^{n-1})\otimes A)\}
\end{align*}
By homotopy invariance, $csr(C(\D^n)\otimes A) = csr(A)$, so the result now follows by induction and Lemma \ref{lem:surj_restriction}
\end{proof}

We now turn our attention to a particularly tractable class of C*-algebras. Let $\mathcal{F}$ be the class of C*-algebras $A$ such that the map $\theta^m_A : GL_{m-1}(A)\to GL_m(A)$ induces a weak homotopy equivalence for all $m\geq 2$. The following algebras are known to be in $\mathcal{F}$.
\begin{itemize}
\item\cite{jiang} If $\mathcal{Z}$ denotes the Jiang-Su algebra, then $A\otimes \mathcal{Z} \in \mathcal{F}$ for any C*-algebra $A$. In particular, if $A$ is a separable, approximately divisible C*-algebra, then $A\cong A\otimes \mathcal{Z}$, so $A\in \mathcal{F}$. 
\item\cite{rieffel2} If $A$ is an irrational rotation algebra, then $A\in \mathcal{F}$.
\item\cite{thomsen} If $\mathcal{O}_n$ denotes the Cuntz algebra, then $A\otimes \mathcal{O}_n \in \mathcal{F}$ for any C*-algebra $A$.
\item\cite{thomsen} If $A$ is an infinite dimensional simple AF-algebra, then $A\otimes B\in \mathcal{F}$ for any C*-algebra $B$.
\item\cite{zhang} If $A$ is a purely infinite, simple C*-algebra, and $p$ any non-zero projection of $A$, then $pAp\in \mathcal{F}$.
\end{itemize}
Note that $A\in \mathcal{F}$ if and only if $\pi_n(Lc_m(A)) = 0$ for all $n\geq 0,$ and $m\geq 2$. Furthermore, if $A\in \mathcal{F}$, then, for all $n\geq 0$ and $m\geq 1$,
\[
\pi_n(GL_m(A)) \cong \begin{cases}
K_1(A) &: n \text{ even } \\
K_0(A) &: n \text{ odd}
\end{cases}
\]
In particular, the natural map $GL_1(A)/GL_1^0(A)\to K_1(A)$ is an isomorphism.
\begin{thm}\label{thm:gsr_csr_cxa}
Let $A\in \mathcal{F}$, and let $X$ be a compact Hausdorff space, then
\begin{align*}
gsr(C(X)\otimes A) &= gsr(A), \text{ and } \\
csr(C(X)\otimes A) &= \begin{cases}
csr(A) &: \text{ if } csr(A)\geq 2 \\
1 \text{ or } 2 &: \text{ if } csr(A) =1
\end{cases}
\end{align*}
\end{thm}
\begin{proof} We first consider the connected stable rank: Since $A \in \mathcal{F}$, $\pi_n(Lc_m(A)) = 0$ for all $n\geq 0$ and $m\geq 2$. Since $Lc_m(A)$ is an open subset of a normed linear space \cite[Section 1]{rieffel2}, it is homotopy equivalent to a CW-complex \cite[Chapter IV, Corollary 5.5]{lundell}. By Whitehead's theorem \cite[Theorem 4.5]{hatcher}, it follows that $Lc_m(A)$ is contractible. Therefore, if $m\geq \max\{2,gsr(A)\}$, then $Lg_m(A) = Lc_m(A)$ is contractible. Identifying $Lg_m(C(X)\otimes A)$ with $C(X,Lg_m(A))$, we see that $\pi_0(Lg_m(C(X)\otimes A)) = [X,Lg_m(A)]$ is trivial. Thus,
\[
csr(C(X)\otimes A)\leq \max\{2,gsr(A)\}
\]
Now the result follows from the fact that $gsr(A) \leq csr(A)\leq csr(C(X)\otimes A)$. \\

For the general stable rank: By the first part of the argument, we have
\[
gsr(A) \leq gsr(C(X)\otimes A)\leq csr(C(X)\otimes A)\leq \max\{2,gsr(A)\}
\]
If $gsr(A)\geq 2$, there is nothing to prove. If $gsr(A) = 1$, then $A$ must be stably finite, and hence $C(X)\otimes A$ is finite. Since $gsr(C(X)\otimes A)\leq 2$, it must be that $gsr(C(X)\otimes A) = 1$ by Remark \ref{rem:properties} (8). This completes the proof.
\end{proof}

\begin{ex}\label{ex:gsr_csr_trivial}
Some examples illustrate our results:
\begin{enumerate}
\item If $A\in \mathcal{F}$ and $csr(A)=1$, then it is possible that $csr(C(X)\otimes A) = 2$, depending on $X$. For instance, if $A$ is a simple, infinite dimensional, unital AF-algebra, then $csr(A) = 1$. Taking $X = \T$, we see that $K_1(\T A) \cong K_0(A)\oplus K_1(A) \neq 0$ because $A$ is stably finite. Hence, $csr(\T A)=2$ by Remark \ref{rem:properties} (9).
\item If $A$ is an irrational rotation algebra, then $tsr(A) = 1$ and $K_1(A)\neq 0$, so $gsr(A) = 1$ and $csr(A) = 2$ by Remark \ref{rem:properties}, (9) and (10). Since $A$, and hence $C(X)\otimes A$, is finite, it follows that
\[
gsr(C(X)\otimes A) =1 \text{ and } csr(C(X)\otimes A) = 2
\]
for any compact Hausdorff space $X$. This was proved by Rieffel \cite[Proposition 2.5, 2.7]{rieffel} in the case where $X = \T^k$. In fact, these were crucial in proving that $A\in \mathcal{F}$.
\item If $A$ is a Kirchberg algebra, then $A\in \mathcal{F}$ by \cite{zhang}. Furthermore, it was proved by Xue \cite{xue2} that $gsr(A) = csr(A) = 2$ if and only if $A$ has the IBN property (Otherwise $gsr(A) = csr(A) = +\infty$). So if $A$ is a Kirchberg algebra with the IBN property, we can conclude that
\[
gsr(C(X)\otimes A) = csr(C(X)\otimes A) = 2
\]
for any compact Hausdorff $X$. In particular, this is true for $A = \mathcal{O}_{\infty}$.
\item If $A$ is a C*-algebra of real rank zero, then it has been proved in \cite[Lemma 2.2]{lin} that $\inj_0(A) = 1$. Hence, it follows from Theorem \ref{thm:csxa} that $gsr(\T A) = gsr(A)$. This is precisely the argument in \cite[Proposition 3.1]{xue}.
\end{enumerate}
\end{ex}


\section{Examples and Calculations}

We now turn to a few examples that have informed this investigation. 

\subsection{Commutative C*-algebras}

If $X$ and $Y$ are two compact Hausdorff spaces and $X\vee Y$ denotes their wedge sum, then $C(X\vee Y) \cong C(X)\oplus_{\C} C(Y)$ where the maps $C(X)\to \C$ and $C(Y)\to \C$ are the evaluation maps at the common base point. Hence, we get the following satisfying corollary to Theorems \ref{thm:gsr_pullback} and \ref{thm:csr_pullback}.
\begin{cor}\label{cor:gsr_wedge}
For any two compact Hausdorff spaces $X$ and $Y$,
\begin{equation*}
\begin{split}
gsr(C(X\vee Y)) &= \max\{gsr(C(X)),gsr(C(Y))\}, \text{ and }\\
csr(C(X\vee Y)) &= \max\{csr(C(X)), csr(C(Y))\}
\end{split}
\end{equation*}
\end{cor}
\begin{proof}
For the general stable rank: The inclusion map $\iota: X\hookrightarrow X\vee Y$ induces a surjection $\iota^{\ast} : C(X\vee Y)\to C(X)$. Furthermore, the `pinching' map $P: X\vee Y\to X$ that pinches $Y$ to the base point has the property that $\iota^{\ast} \circ P^{\ast} = \text{id}_{C(X)}$. So it follows from Remark \ref{rem:properties}, (5) that $gsr(C(X\vee Y))\geq gsr(C(X))$. By symmetry, the same true for $Y$, so
\[
gsr(C(X\vee Y))\geq \max\{gsr(C(X)), gsr(C(Y)\}
\]
Now observe that $K_1(\C) = 0$, and $\inj_0(\C) = 1$ so the result follows from Theorem \ref{thm:gsr_pullback}. \\

For the connected stable rank: The same argument as above shows that
\begin{align*}
\max\{csr(C(X)),csr(C(Y))\} &\leq csr(C(X\vee Y)) \\
&\leq \max\{csr(C(X)),csr(C(Y)),2\}
\end{align*}
where the second inequality follows from Theorem \ref{thm:csr_pullback} and the fact that $\surj_1(\C) = 2$. Thus, if $\max\{csr(C(X)),csr(C(Y))\} \geq 2$, then the conclusion follows. Suppose $csr(C(X)) = csr(C(Y)) = 1$. We must conclude that $csr(C(X\vee Y)) = 1$. By the above inequality, we know that $csr(C(X\vee Y))\leq 2$. Hence, it suffices to show that $Lg_1(C(X\vee Y))$ is connected. However,
\[
\pi_0(Lg_1(C(X\vee Y))) = \pi_0(C(X\vee Y, Lg_1(\C)) \cong [X\vee Y, \T]
\]
Since $csr(C(X)) = csr(C(Y)) = 1$, we know that $[X,\T]$ and $[Y,\T]$ are both trivial. Since $gsr(C(\T)) = csr(\C) = 1$, it follows from Lemma \ref{lem:forgetful} that $[X,\T]_{\ast}$ and $[Y,\T]_{\ast}$ are both trivial as well. If $f:X\vee Y\to \T$ is a map based at the identity, then $f\circ \iota : X\to \T$ must be null-homotopic. Similarly, if $j :Y\hookrightarrow X\vee Y$ denotes the inclusion map, then $f\circ j$ is also null-homotopic. Furthermore, the homotopies may be chosen to preserve the common base point, so we may paste the two homotopies together to conclude that $f$ is null-homotopic. Hence, $[X\vee Y,\T]_{\ast}$ is trivial. Once again, by Lemma \ref{lem:forgetful}, we conclude that $[X\vee Y, \T]$ is also trivial. Hence, $Lg_1(C(X\vee Y))$ is connected, whence $csr(C(X\vee Y)) = 1$ as required.
\end{proof}

Our next goal is determining $gsr(C(\T^d))$. To begin with, we have the following observation.

\begin{cor}\label{cor:t_times_x}
If $X$ is a compact Hausdorff space, then
\[
gsr(C(\T\times X)) = \max\{gsr(C(X)), gsr(C(\Sigma X))\}
\]
\end{cor}
\begin{proof}
Note that $C(\T \times X) = \T A$ where $A = C(X)$, so by Theorem \ref{thm:gsr_csr_cxa},
\[
gsr(C(\T\times X)) = \max\{gsr(C(X)), \inj_0(C(X))\}
\]
By Lemma \ref{lem:inj_cx}, $\inj_0(C(X)) = \inj_X(\C)$, so the result follows from Equation \ref{eqn:gsr_sx}.
\end{proof}

Recall that a space $X$ is said to homotopically dominate $Y$ if there is are maps $P:X\to Y$ and $f:Y\to X$ such that $P\circ f\simeq \text{id}_Y$. If this happens, then $C(X)$ homotopically dominates $C(Y)$, so it follows from Theorem \ref{thm:homotopy_dom} that $gsr(C(X))\geq gsr(C(Y))$.

\begin{lemma}\label{lem:td_dominates}
If $X = \prod_{i=1}^k \S^{n_i}$, then $\Sigma X$ homotopically dominates $\S^{n+1}$ where $n = \sum_{i=1}^k n_i$. In particular, $\Sigma \T^n$ homotopically dominates $\S^{n+1}$
\end{lemma}
\begin{proof}
We claim that
\[
\Sigma X \simeq \S^{n+1}\vee M
\]
for some manifold $M$ of dimension $\leq n$. To see this, we proceed by induction on $k$. It is clearly true if $k=1$, so let $Y = \prod_{i=1}^{k-1} \S^{n_i}$ and assume $\Sigma Y \simeq \S^{\ell+1}\vee N$,
where $\ell = \sum_{i=1}^{k-1} n_i$ and $N$ is a manifold of dimension $\leq \ell$. Then by \cite[Proposition 4I.1]{hatcher},
\begin{align*}
\Sigma X = \Sigma (Y\times \S^{n_k}) &\simeq \Sigma Y \vee \S^{n_k+1}\vee \Sigma (Y\wedge \S^{n_k}) \\
&\simeq \S^{\ell+1}\vee N \vee \S^{n_k+1} \vee \Sigma^{n_k}(\S^{\ell+1}\vee N)) \\
&\simeq \S^{\ell+1}\vee N \vee \S^{n_k+1}\vee \Sigma^{n_k}(N)\vee \S^{\ell+n_k+1} \simeq M\vee \S^{n+1}
\end{align*}
where $M = \S^{\ell+1}\vee N \vee \S^{n_k+1}\vee \Sigma^{n_k}(N)$. Note that
\[
\dim(M) \leq \max\{\ell+1,\ell,n_k+1, n_k+\ell\} \leq n_k + \ell = n
\]
This proves the claim. So we get a map $P : \Sigma X\to M\vee \S^{n+1} \to \S^{n+1}$ that `pinches' $M$ to a point, and a map $f: \S^{n+1} \to \S^{n+1}\vee M \to \Sigma X$ by composing the homotopy equivalence with the natural map $\S^{n+1}\to \S^{n+1}\vee M$. Note that $P_{\ast} : H_{n+1}(\Sigma X)\to H_{n+1}(\S^{n+1})$ is an isomorphism because $\dim(M)\leq n$, and $f_{\ast} : H_{n+1}(\S^{n+1}) \to H_{n+1}(\Sigma X)$ is also an isomorphism. Hence,
\[
(P\circ f)_{\ast} : H_{n+1}(\S^{n+1}) \to H_{n+1}(\S^{n+1})
\]
is an isomorphism. Since both $P$ and $f$ are orientation-preserving, it follows that $P \circ f$ has degree 1, and so $P \circ f\simeq \text{id}_{\S^{n+1}}$ as required.
\end{proof}

The following is an answer to a question posed by Nica \cite[Problem 5.8]{nica2}. Before we begin, we observe that if $X$ is a compact Hausdorff space whose covering dimension is $\leq 4$, then Nica has shown \cite[Proposition 5.5]{nica2} that $gsr(C(X)) = 1$. The point of this next example, thus, is using the previous lemma to compare $gsr(C(\T^d))$ and $gsr(C(\S^d))$ for $d\geq 5$. \\

To put this in perspective, if $X$ is a compact Hausdorff space of covering dimension $\leq n$, then $csr(C(X))\leq \lceil n/2\rceil + 1$ by \cite[Corollary 2.5]{nistor} (See also Corollary \ref{cor:nistor}). Furthermore, Nica has shown \cite[Theorem 5.3]{nica2} that this upper bound is attained provided the top cohomology group $H^{\text{odd}}(X)$ is non-vanishing. In particular, this implies that, for all $d\geq 1$,
\[
csr(C(\T^d)) = \left\lceil \frac{d}{2}\right\rceil + 1
\]

\begin{ex}\label{ex:gsr_td}
\[
gsr(C(\T^d)) = \begin{cases}
1 &: d\leq 4 \\
\lceil \frac{d}{2}\rceil + 1 &: d>4
\end{cases}
\]
\end{ex}
\begin{proof}
For $d\leq 4$, the result follows from the preceding discussion. For $d\geq 5$, we know that
\[
gsr(C(\T^d))\leq csr(C(\T^d))\leq \left\lceil \frac{d}{2}\right\rceil + 1
\]
so it suffices to prove the reverse inequality. We proceed by induction on $d$. For $d=5$, by Corollary \ref{cor:t_times_x} and Lemma \ref{lem:td_dominates},
\[
gsr(C(\T^5)) \geq gsr(C(\Sigma \T^4))) \geq gsr(C(\S^5))
\]
and $gsr(C(\S^5)) = 4$ by Example \ref{ex:gsr_suspension}. For $d\geq 6$, by induction
\begin{align*}
gsr(C(\T^d)) &= \max\{gsr(C(\T^{d-1}), gsr(C(\Sigma \T^{d-1})\} \\
&\geq \max\left\{\left\lceil \frac{d-1}{2}\right\rceil + 1, gsr(C(\S^d))\right\}
\end{align*}
Once again the result follows from Example \ref{ex:gsr_suspension}.
\end{proof}

\subsection{NonCommutative CW-Complexes}

As observed in Subsection \ref{subsection:csr_commutative}, a commutative C*-algebras whose spectrum is a finite CW-complex can be expressed as an (iterated) pullback. Noncommutative CW-complexes (NCCW complexes), first studied by Pedersen \cite{pedersen}, are meant to generalize this idea: A NCCW complex $A_0$ of dimension $0$ is a finite dimensional C*-algebra. A NCCW complex $A_k$ of dimension $k$ is described by a pullback
\[
\xymatrix{
A_k\ar[r]\ar[d] & A_{k-1}\ar[d] \\
C(\D^k)\otimes F_k\ar[r]^{\gamma} & C(\S^{k-1})\otimes F_k
}
\]
where $F_k$ is a finite dimensional C*-algebra, $A_{k-1}$ is an NCCW complex of dimension $(k-1)$, and $\gamma$ is the restriction map. If $F$ is a finite dimensional C*-algebra, then it follows from Remark \ref{rem:properties} that $csr(F) = 1$. Hence, $csr(A_0) = 1$ and $csr(C(\D^k)\otimes F_k) = csr(F_k) = 1$ by homotopy invariance. If $D = C(\S^{k-1})\otimes F_k$, then by Lemma \ref{lem:surj_restriction},
\[
\max\{\inj_0(D), \surj_1(D)\} \leq \max\{\surj_1(F_k), \surj_k(F_k), \inj_{k-1}(F_k)\}
\]
Write $F_k = \bigoplus_{i=1}^{n_k} M_{\ell_i}(\C)$, then $\surj_1(F_k) = 2$, so computing the right hand side boils down to asking whether, for all $1\leq i\leq n_k$, the map
\begin{equation*}
\begin{split}
\pi_k(GL_{\ell_i(m-1)}(\C))\to \pi_k(GL_{\ell_i m}(\C)) &\text{ is surjective, and} \\
\pi_{k-1}(GL_{\ell_i(m-1)}(\C))\to \pi_{k-1}(GL_{\ell_i m}(\C)) &\text{ is injective}
\end{split}
\end{equation*}
By Bott periodicity, these maps are isomorphisms if $k\leq 2\ell_i(m-1)-1$ (See, for instance, \cite[Page 251-254]{james}). Furthermore, if $k = 2\ell_i(m-1)$, then both conditions are satified because the second map is an isomorphism, and $\pi_k(GL_{\ell_i m}(\C)) = 0$. So if $d_k = \min\{\ell_i : 1\leq i\leq j_k\}$, then we have
\[
\max\{\inj_0(D), \surj_1(D)\} \leq \left\lceil \frac{k}{2d_k}\right\rceil + 1
\]
The following estimate is thus a corollary of Theorem \ref{thm:csr_pullback}
\begin{thm}\label{thm:nccw_complex}
Let $A_n$ be an NCCW complex of topological dimension atmost $n$ whose structure can be described as above. Then
\[
csr(A_n) \leq \max_{1\leq k\leq n} \left\lbrace \left\lceil \frac{k}{2d_k}\right\rceil + 1\right\rbrace \leq \left\lceil \frac{n}{2}\right\rceil + 1
\]
\end{thm}

A special case of this theorem is that of a commutative C*-algebra whose spectrum is a finite CW-complex. As in the proof of Theorem \ref{thm:csr_finite_dim}, by passing to inductive limits we obtain yet another proof of a result due to Nistor.
\begin{cor}\cite[Corollary 2.5]{nistor}\label{cor:nistor}
If $X$ is a compact Hausdorff space of dimension atmost $n$, then $csr(C(X))\leq \lceil n/2\rceil + 1$
\end{cor}

\subsection*{Acknowledgements}
The author was supported by SERB (Grant No. YSS/2015/001060). The author would also like to thank Prof. Thomsen for kindly providing him a copy of \cite{thomsen}. 

\bibliographystyle{plain}
\bibliography{mybib}

\begin{thebibliography}{10}

\bibitem{brown}
Lawrence~G Brown and Gert~K Pedersen.
\newblock Limits and {$C^*$}-algebras of low rank or dimension.
\newblock {\em Journal of Operator Theory}, pages 381--417, 2009.

\bibitem{corach}
Gustavo Corach and Angel~R Larotonda.
\newblock A stabilization theorem for {B}anach algebras.
\newblock {\em Journal of Algebra}, 101(2):433--449, 1986.

\bibitem{freudenthal}
Hans Freudenthal.
\newblock Entwicklungen von {R}{\"a}umen und ihren {G}ruppen.
\newblock {\em Compositio math}, 4(145-234):25, 1937.

\bibitem{hatcher}
Allen Hatcher.
\newblock {\em Algebraic topology. 2002}, volume 606.
\newblock 2002.

\bibitem{jiang}
Xinhui Jiang.
\newblock Nonstable {K}-theory for {$\mathcal{Z}$}-stable {$C^*$}-algebras.
\newblock {\em arXiv preprint math/9707228}, 1997.

\bibitem{lin}
Huaxin Lin.
\newblock Approximation by normal elements with finite spectra in
  {$C^*$}-algebras of real rank zero.
\newblock {\em Pacific Journal of Mathematics}, 173(2):443--489, 1996.

\bibitem{james}
Albert~T Lundell.
\newblock Concise tables of {J}ames numbers and some homotopy of classical lie
  groups and associated homogeneous spaces.
\newblock In {\em Algebraic Topology Homotopy and Group Cohomology}, pages
  250--272. Springer, 1992.

\bibitem{lundell}
Albert~T Lundell and Stephen Weingram.
\newblock {\em The topology of CW complexes}.
\newblock Springer Science \& Business Media, 2012.

\bibitem{magurn}
Bruce~A Magurn.
\newblock {\em An algebraic introduction to {K}-theory}, volume~87.
\newblock Cambridge University Press, 2002.

\bibitem{mardesic}
Sibe Marde{\v{s}}i{\'c} et~al.
\newblock On covering dimension and inverse limits of compact spaces.
\newblock {\em Illinois Journal of Mathematics}, 4(2):278--291, 1960.

\bibitem{milnor}
John Milnor.
\newblock {\em Introduction to Algebraic {K}-Theory.(AM-72)}, volume~72.
\newblock Princeton University Press, 2016.

\bibitem{nagisa}
Masaru Nagisa, Hiroyuki Osaka, and N~Christopher Phillips.
\newblock Ranks of algebras of continuous {$C^*$}-algebra valued functions.
\newblock In {\em Canad. J. Math}. Citeseer, 2001.

\bibitem{nica2}
Bogdan Nica.
\newblock Homotopical stable ranks for {B}anach algebras.
\newblock {\em Journal of Functional Analysis}, 261(3):803--830, 2011.

\bibitem{nistor}
Victor Nistor.
\newblock Stable range for tensor products of extensions of {$\mathcal{K}$} by
  {$C(X)$}.
\newblock {\em Journal of Operator Theory}, pages 387--396, 1986.

\bibitem{pedersen}
Gert~K Pedersen.
\newblock Pullback and pushout constructions in {$C^*$}-algebra theory.
\newblock {\em Journal of functional analysis}, 167(2):243--344, 1999.

\bibitem{rieffel}
Marc~A Rieffel.
\newblock Dimension and {S}table rank in the {K}-theory of {$C^*$}-algebras.
\newblock {\em Proceedings of the London Mathematical Society}, 3(2):301--333,
  1983.

\bibitem{rieffel2}
Marc~A Rieffel.
\newblock The homotopy groups of the unitary groups of non-commutative tori.
\newblock {\em Journal of Operator Theory}, pages 237--254, 1987.

\bibitem{rieffel3}
Marc~A Rieffel.
\newblock Projective modules over higher-dimensional noncommutative tori.
\newblock {\em Canad. J. Math}, 40(2):257--338, 1988.

\bibitem{schochet}
Claude Schochet.
\newblock Topological methods for {$C^*$}-algebras: {III}. {A}xiomatic
  homology.
\newblock {\em Pacific Journal of Mathematics}, 114(2):399--445, 1984.

\bibitem{sheu}
AJ-L Sheu.
\newblock A cancellation theorem for projective modules over the group
  {$C^*$}-algebras of certain nilpotent {L}ie groups.
\newblock {\em Canad. J. Math.}, 39:365--427, 1987.

\bibitem{thomsen}
Klaus Thomsen.
\newblock Nonstable {K}-theory for operator algebras.
\newblock {\em K-theory}, 4(3):245--267, 1991.

\bibitem{xue2}
Yifeng Xue.
\newblock The connected stable rank of the purely infinite simple
  {$C^*$}-algebras.
\newblock {\em Proceedings of the American Mathematical Society},
  127(12):3671--3676, 1999.

\bibitem{xue}
Yifeng Xue.
\newblock The general stable rank in nonstable {K}-theory.
\newblock {\em The Rocky Mountain Journal of Mathematics}, pages 761--775,
  2000.

\bibitem{zhang}
Shuang Zhang.
\newblock Matricial structure and homotopy type of simple {$C^*$}-algebras with
  real rank zero.
\newblock {\em Journal of Operator Theory}, pages 283--312, 1991.

\end{thebibliography}
\end{document}